\documentclass[11pt]{amsart}
\usepackage{mathrsfs}
\usepackage{mathdots}
\usepackage{amsmath,amsthm,amsfonts,amssymb,mathrsfs,amscd,mathtools}
\usepackage{graphicx}
\usepackage{dsfont}
\usepackage{hyperref}
\usepackage{xcolor}
\usepackage{pgfplots}
\usepackage{tikz-cd}
\usepackage[mathcal]{euscript}
\usepackage{subcaption}
\usepackage[english]{babel}
\usepackage{url}
\hypersetup{
    colorlinks=true,
    linkcolor=blue, 
    citecolor=magenta,
}
\usepackage{extarrows}
\newcommand{\C}{\mathbb{C}}
\newcommand{\R}{\mathbb{R}}
\newcommand{\N}{\mathbb{N}}
\newcommand{\Z}{\mathbb{Z}}
\newcommand{\Q}{\mathbb{Q}}
\renewcommand{\P}{\mathbb{P}}
\newcommand{\ohne}{\smallsetminus}
\renewcommand{\O}{\mathcal{O}}
\renewcommand{\P}{\mathbb{P}}
\DeclareMathOperator{\Spec}{Spec}
\DeclareMathOperator{\pr}{pr}
\DeclareMathOperator{\id}{id}
\newcommand{\X}{\mathscr{X}}
\renewcommand{\SS}{\mathscr{S}}
\newcommand{\U}{\mathcal{U}}
\newcommand{\V}{\mathcal{V}}
\renewcommand{\mod}{\mathrm{mod}}
\newcommand{\Pic}{\mathrm{Pic}}

\renewcommand{\L}{\mathcal{L}}

\newcommand{\Sbar}{\overline{S}}
\newcommand{\NN}{\mathscr{N}}
\DeclareMathOperator{\colim}{colim}		
\newcommand{\Div}{\mathrm{Div}}
\newcommand{\divisor}{\mathrm{div}}
\newcommand{\an}{\mathrm{an}}
\newcommand{\scrB}{\mathscr{B}}
\newcommand{\scrD}{\mathscr{D}}
\newcommand{\LL}{\mathscr{L}}
\newcommand{\PP}{\mathcal{P}}
\newcommand{\Kbar}{\overline{K}}

\newcommand{\Qbar}{\overline{\mathbb{Q}}}
\calclayout

\newtheorem{thm}{Theorem}[section]
\newtheorem{cor}[thm]{Corollary}
\newtheorem{prop}[thm]{Proposition}
\newtheorem{lem}[thm]{Lemma}

\newtheorem{defp}[thm]{Definition and Proposition}
\newtheorem{theorem}{Theorem}

\numberwithin{equation}{subsection}

\makeatletter
\renewcommand{\theequation}{%
	\thesection.%
	\ifnum\value{subsection}=0
	0%
	\else
	\Alph{subsection}%
	\fi
	.\arabic{equation}%
}
\makeatother

\theoremstyle{definition}
\newtheorem{void}[thm]{}

\newtheorem{defn}[thm]{Definition}

\newtheorem{ex}[thm]{Example}

\theoremstyle{remark}
\newtheorem{rem}[thm]{Remark}
\newtheorem{conv}[thm]{Convention}

\pgfplotsset{compat=1.9}
\title{Adelic line bundles over a Néron model}

\author{Zhelun Chen}
\date{\today}
\address{Mathematical Institute, Leiden University, Einsteinweg 55, 2333 CC Leiden, The Netherlands}
\email{zl.chen1729@gmail.com}

\begin{document}

\begin{abstract}
Motivated by variation problems for the Néron--Tate height pairing, we construct an adelic Poincaré bundle on the Néron model of an abelian scheme over a curve. This extends the adelic Poincaré bundle of Yuan--Zhang \cite{YZ21}, which encodes the fiberwise Néron--Tate height pairing.
\end{abstract}

\maketitle

\tableofcontents
\section{Introduction}
\subsection{First motivation}
Let $K$ be a number field. Let $S$ be a smooth curve over $K$, let $\eta\in S$ be its generic point, and let $S\subset \Sbar$ be the smooth projective compactification over $K$. Consider an abelian scheme $A/S$ and the Néron--Tate height $\hat{h}$ associated to a symmetric relatively ample line bundle on $A/S$. In 1983, Silverman proved the following remarkable result, which relates the natural height functions attached to $A/S$; \emph{cf}.\ \cite[Theorem~B]{Sil83}:
\begin{thm}[Silverman]\label{Silverman's theorem}
For every section $P\in A(S)$, one has the  asymptotic description of the fiberwise Néron--Tate height:
	\[
	\lim\limits_{s\in S(\Kbar), h_{\Sbar}(s)\to \infty} \frac{\hat{h}(P_s)}{h_{\Sbar}(s)}=\hat{h}(P_\eta)\in \Q,
	\] where $h_{\Sbar}$ is any Weil height on $\Sbar$ associated to a degree-one line bundle on $\Sbar$.
\end{thm}
Here $\hat{h}(P_\eta)$ denotes the geometric Néron--Tate height of $P_\eta$, which can be computed as  the degree of a certain $\Q$-line bundle on $\Sbar$ (\emph{cf}.\ Remark \ref{height and height pairing}). An interesting application is already given in \cite[Theorem~C]{Sil83}, where Silverman describes the locus of closed points $s\in S$ for which the specialization map
\[
A(S)\to A_s(\Qbar),\quad P\mapsto P_s
\] is injective, see Corollary \ref{uniform torsion, 1d} at the end of \S\ref{Application to Silverman's Limit Theorem}.

When $A/S$ is a family of elliptic curves, Tate \cite{Tate83} proved the stronger result that, for every $P\in A(S)$, the function
\begin{equation}\label{fiberwise height}
h_P\colon S(\overline{\Q})\to \R,\quad s\mapsto \hat{h}(P_s)
\end{equation} 
differs from a Weil height function on  $\Sbar$  by a bounded function. Tate's result has been generalized by several authors, see Remark \ref{Silverman-Tate, literature}. In particular, the recent theory of adelic line bundles developed by Yuan--Zhang produces a height function $\hat{h}_P$ on $S$ that \emph{exactly} recovers \eqref{fiberwise height}, see \cite[Lemma~6.2.1]{YZ21}. We strengthen this result by constructing a canonical height function $\hat{h}_{P,\Sbar}$ on $\Sbar$ whose restriction to $S$ is $\hat{h}_P$.

In fact, one can consider more generally the Néron--Tate height pairing of sections $P\in A(S)$ and $Q\in A^\vee(S)$, where $A^\vee/S$ denotes the dual abelian scheme. Let $\NN(A)^0/\Sbar$ resp.\ $\NN(A^\vee)^0/\Sbar$ be the identity components of the Néron models of $A/S$ resp.\ $A^\vee/S$. The  Poincaré bundle $\PP\in \Pic(A\times_S A^\vee)$, which by definition is rigidified along the identity section, extends to a rigidified line bundle $\overline{\PP}\in \Pic(\NN(A)^0\times_{\Sbar}\NN(A^\vee)^0)$, see Proposition~\ref{prolongation of  Poincaré  bundle}; we also call $\overline{\PP}$ the \textbf{extended Poincaré bundle}. There is a positive integer $n$ such that~$nP,\, nQ$ extend to sections $\overline{P}\in \NN(A)^0(\Sbar),\, \overline{Q}\in \NN(A^\vee)^0(\Sbar)$ respectively, \emph{cf}.\ Remark~\ref{Rmk: order of component group is finite}.
\begin{theorem}[Theorems \ref{adelic prolongation of Poincare bundle} \& \ref{adelic refinement of Silverman's limit thm}]\label{Theorem A}
	The extended Poincaré bundle $\overline{\PP}$ admits an integrable adelic structure $\widehat{\overline{\PP}}$, and the pullback 
	\[
\widehat{\overline{\mathcal{M}}}\coloneq	\frac{1}{n^2}(\overline{P},\overline{Q})^*\widehat{\overline{\PP}},
	\] which is an integrable adelic line bundle on $\Sbar$, encodes the fiberwise Néron--Tate height pairing of $P, Q$ over $S$, that is, the height $h_{\widehat{\overline{\mathcal{M}}}}$ associated to the adelic line bundle $\widehat{\overline{\mathcal{M}}}$ (see \cite[\S 5.3.1]{YZ21}) satisfies
	\[
	h_{\widehat{\overline{\mathcal{M}}}}(s)=\langle P_s,Q_s\rangle_{A_s}
	\] for all $s\in S(\Qbar)$.
\end{theorem}
See Remark \ref{height and height pairing} for the relation between the Néron--Tate height and the Néron--Tate height pairing. We also record that A.~Carney \cite{Carney} constructs an adelic line bundle on~$\Sbar$ encoding the Néron--Tate height (over $S$). However, it seems unclear to us whether this adelic line bundle can be realized as the pullback of an adelic version of the relevant extended Poincaré bundle. In fact, our construction is trigger by \cite[Remark  4.1]{Carney}.

\subsection{Second motivation}
Another motivation for the present work comes from the study of variations of local archimedean height pairings. Let $S$ be a smooth complex curve, and let $X/S$ be a smooth projective family of complex varieties of dimension $d$. Let $Z$ and $W$ be flat algebraic cycles on $X$. Assume that, for every $s\in S(\C)$, the  codimensions of the fibers $Z_s$ and $W_s$ in $X_s$ add up to $d+1$, and that $Z$ and $W$ have disjoint supports. Then the fiberwise archimedean height pairing
\begin{equation}\label{fiberwise arch height pairing}
	S(\C)\to \R, \quad \langle Z_s, W_s\rangle_\infty
\end{equation} is well-defined, see \cite[\S 1]{Bost} for a comprehensive introduction to the  analytic construction of this pairing.

One can ask about the asymptotic behaviour of the function \eqref{fiberwise arch height pairing}, in analogy with the global Néron--Tate height pairing discussed in the previous subsection. When $Z$ and $W$ are fiberwise homologically trivial, Hodge-theoretic approaches to such questions can be found in \cite{Hain90, Pea06, BP19, Nakayama_AsymptoticHeight}.  When~$Z$ and $W$ are fiberwise algebraically trivial, these questions were studied  in our earlier work \cite{CZL25}. In essence, algebraic triviality allows one to transport these questions to corresponding questions about the  (geometric) Néron--Tate height pairing over a complex function field.

In this manuscript, we provide a proof of the result announced in \cite[Theorem~7.15]{CZL25}. Let $A/S$ be a complex abelian scheme. Let $(\PP,\|\cdot\|)\in \overline{\Pic}(A\times_S A^\vee)$ be the Poincaré bundle, endowed with the smooth admissible metric normalized by the rigidification (see Proposition \ref{admissible metric on line bundle on abelian scheme}).  Let $S\subset \Sbar$ be a smooth partial compactification, and let $\NN(A)^0/\Sbar,\, \NN(A^\vee)/\Sbar$ be the identity components of the Néron models of $A/S,\, A^\vee/S$ respectively. Let $\overline{\PP}\in \Pic(\NN(A)^0\times_{\Sbar}\NN(A^\vee)^0)$ be the (algebraic) extended  Poincaré bundle. By Remark \ref{continuous extension of metrized line bundle} and Theorem \ref{adelic extension summary, analytic case}, we obtain
\begin{theorem}\label{Theorem B}
The line bundle $\overline{\PP}$ admits an integrable analytic adelic structure. In particular, the admissible metric $\|\cdot\|$ extends to a continuous admissible metric on $\overline{\PP}$.
\end{theorem}
\begin{rem}
Holmes--de Jong \cite{HdJ15} prove the ``in particular'' part of Theorem~\ref{Theorem B} in the case where $A/S$ is principally polarized and has semistable reduction over $\Sbar$, using an explicit computation of theta functions. In contrast, thanks to the framework of Yuan--Zhang \cite{YZ21}, our proof is somewhat formal and does not involve explicit analytic arguments.
\end{rem}
\subsection{Proof strategy}
In \cite[\S 6.1]{YZ21}, Yuan--Zhang construct an integrable adelic extension of the Poincaré bundle $\PP \in \Pic(A \times_S A^\vee)$ for an abelian scheme $A/S$ over a quasi-projective variety $S$ of \emph{arbitrary} dimension. Assuming $\dim S=1$, and working over either the archimedean valued field $\C$ or a number field, we use similar ideas to generalize their construction to the extended Poincaré bundle over the Néron model.
\subsection{Organization}
We recall the construction of the Néron--Tate height pairing via the (admissible metrized) Poincaré bundle in \S\ref{Sect:Néron--Tate Height Pairing} and review Yuan--Zhang's theory of adelic line bundles in \S\ref{The Notion of Adelic Line Bundles}. We first prove Theorem~\ref{Theorem B} in \S\ref{Sect:Complex-analytic Case}; the proof of adelic extension part of Theorem~\ref{Theorem A}, given in \S\ref{Arithmetic Case}, is similar. We deduce Theorem~\ref{Theorem A} and explain how it recovers  Theorem~\ref{Silverman's theorem} in the final  \S\ref{Application to Silverman's Limit Theorem}.

\subsection{Acknowledgements}
I would like to thank Robin de Jong for suggesting this problem to me. I am especially grateful to David Holmes for his detailed comments on a previous version of the manuscript.  I acknowledge the financial support of the CSC--Leiden University Scholarship No. 202208080111.

\subsection{Conventions}\label{General Conventions}
Throughout the manuscript, we adopt the following conventions.
\begin{enumerate}
			\item 	(conventions for varieties) Let $k$ be a field. By an \textbf{algebraic scheme} over  $k$, or an algebraic $k$-scheme, we mean a separated scheme of finite type over $\Spec(k)$. An \textbf{(algebraic) variety} over~$k$ is a geometrically integral algebraic scheme. An \textbf{(algebraic) curve} is a 1-dimensional variety.

			\item 	(Landau symbols) Let $f, g\colon X\to \R$ be two real-valued functions on a set $X$. Then $f=O(g)$ means that there is a (unspecified) $C\in \R_{>0}$ such that $|f(x)|\leq Cg(x)$ for all $x\in X$. If $h\colon X\to \R$ is another function, we write $f=h+O(g)$ to mean $f-h=O(g)$. For example, $f=g+O(1)$ means that the difference of the functions $f$ and $g$ is bounded.

			\item We use the additive notation for Picard groups. For example, if $L$ and $M$ are line bundles, $L-M$ means $L\otimes M^\vee$ (and $M^\vee$ denotes the dual line bundle).
			
			\item ($\Q$-line bundles) We  use the formalism of $\Q$-line bundles as explained in \cite[Definition 2.1]{BHdJ18}. In practice, we  think of and denote a $\Q$-line bundle somewhat informally in the form of $\frac{1}{r}L$, where $r$ is a positive integer and $L$ is  an ordinary line bundle. An isomorphism $\frac{1}{r_1}L_1\to \frac{1}{r_2}L_2$ of $\Q$-line bundles can be represented by an isomorphism $\frac{a}{r_1}L_1\to \frac{a}{r_2}L_2$ of ordinary line bundles, where~$a$ is a positive integer  such that $\frac{a}{r_1}$ and $\frac{a}{r_2}$ are positive integers (e.g.\ $a\coloneq r_1r_2$). 
			
			\item	Let $\mathbf{P}$ be one of the following positivity properties of line bundles: $\mathbf{P}$=``ample'', ``nef'', ``big'' or ``pseudoeffective''. Then we say a $\Q$-line bundle $L$ has property $\mathbf{P}$, if some positive multiple $mL$ such that $mL$ is a usual line bundle and $mL$ has property $\mathbf{P}$.
			
			\item (convention for  (quasi)-projective morphisms)  A \emph{finite type} morphism $X\to S$ of schemes is \textbf{projective} resp.\ \textbf{quasi-projective}, if there exists an $S$-closed immersion resp.\  locally closed $S$-immersion $X\hookrightarrow \P^n_S$ into some projective space over $S$, \emph{cf}.\ \cite[Summary 13.71]{GW}. 
\end{enumerate}
\section{Néron--Tate height pairing}\label{Sect:Néron--Tate Height Pairing}
We recall the construction of the Néron--Tate height pairing, following Moret-Bailly \cite{MB_adm, Ast129}. 
\subsection{Rigidified line bundles}
\begin{defn}\label{rigidified line bundle}
	Let $f\colon X\to S$ be a  morphism of schemes and $i\colon S\to X$ be a section of~$f$. A \textbf{rigidified line bundle} along $i$ is a pair $(L,\phi)$, where $L$ is a line bundle on $X$ an $u$ is an isomorphism 
	\[
	\phi\colon i^*L\xrightarrow{\sim} \O_S.
	\] A morphism between two rigidified line bundles $(L,\phi)$ and $(L',\phi')$ is an isomorphism $u\colon L\xrightarrow{\sim} L'$ which is compatible with the rigidifications:
	\[
	\phi'\circ i^*u=\phi\colon i^*L\xrightarrow{\sim}\O_S.
	\] We thus obtain a category (groupoid) $\mathsf{PicRig}(X,i)$ with objects consisting of rigidified line bundles, and morphisms are isomorphisms between rigidified line bundles. The trivial element in $\mathsf{PicRig}(X,i)$ is the trivial line bundle $\O_X$ with the canonical isomorphism $1\colon i^*\O_X=\O_S$.
\end{defn}
For ease of notation, we suppress the rigidification.
\begin{lem}[extension of rigidified line bundles]\label{equivalence of cat. of rigidified line bundle}
	Let $f\colon X\to \Sbar$ be a smooth  morphism  over   a Dedekind scheme\footnote{We always assume Dedekind schemes to be connected.} $\Sbar$. Assume moreover that the fibers of $f$  are connected. Let $i$ be a section of~$f$.  Let $S\subset \Sbar$ be a dense open subscheme. Then the restriction $X_S\subset X$ induces an equivalence of categories
	\[
	\mathrm{Res}_S\colon\mathsf{PicRig}(X,i)\to \mathsf{PicRig}(X_S,i_S),
	\]  where $i_S\colon S\to X_S$ denotes the section induced by $i$.
\end{lem}
\begin{proof}
	See \cite[Lemme 2.8.2.1]{MB_adm}.
\end{proof}

We deduce an important extension result for Poincaré bundles. For a detailed discussion of the fundamental notion of a Poincaré bundle, see \cite[\S 27.41]{GW2}. Recall that, by definition, Poincaré bundles are rigidified along the identity sections of the relevant abelian schemes.

\begin{prop}[prolongation of  the Poincaré  bundle]\label{prolongation of  Poincaré  bundle}
	Let $S\subset \Sbar$ be  a dense open subset of a Dedekind scheme $\Sbar$. Let $A\to S$ be  an abelian scheme, and suppose there are smooth commutative group schemes $G/\Sbar$ resp.\ $G'/\Sbar$ extending $A/S$ resp.\ the dual abelian scheme $A^\vee/S$. Assume both $G$ and $G'$  have \emph{connected} fibers. Then the Poincaré bundle $\PP$ over $A\times_S A^\vee$  admits a \emph{unique} extension  $\overline{\PP}$ over $G\times_{\Sbar} G'$ as a rigidified line bundle along the identity section.
\end{prop}
\begin{rem}\label{symmetry of Poincare bundle}
	For a product $G\times_S G'$ of commutative group schemes over a scheme $S$, we let 
	\[
	[-1]\colon G\times_S G'\to G \times_S G',\quad (x,y)\mapsto (-x,-y)
	\] denote the total multiplication by $-1$. Now use the notation from Theorem \ref{prolongation of  Poincaré  bundle}. The Poincaré  bundle $\PP$ is symmetric, i.e. $[-1]^*\PP\cong \PP$. We note that the extension $\overline{\PP}$ is also symmetric by Lemma  \ref{equivalence of cat. of rigidified line bundle}.
\end{rem}

Proposition \ref{prolongation of  Poincaré  bundle} applies to the following situation.  Let $A/S$ be an abelian scheme and let $\NN(A)^0\to \Sbar$ be the identity component of the Néron model $\NN(A)$ of $A$ over $\Sbar$ (i.e.\ its fibers are the connected components of the identity element of the group schemes). Then the Poicaré bundle $\PP$ extends uniquely to a rigidified symmetric line bundle $\overline{\PP}$ on $\NN(A)^0\times_{\Sbar} \NN(A^\vee)^0$. 
See \cite{BLR} for a standard introduction to Néron models.

\begin{lem}\label{ample extension}
	Let $L$ be a rigidified symmetric and relatively ample w.r.t.\ $A/S$. It extends uniquely to a rigidified symmetric  line bundle $\L$ which is relatively ample w.r.t.\   $\NN(A)^0/\Sbar$.
\end{lem}
\begin{proof}
	As we already observed from Lemma \ref{equivalence of cat. of rigidified line bundle}, there is a unique rigidified symmetric line bundle $\L$ on $\NN(A)^0$ that extends $L$. By \cite[Théorème VIII.2(2)]{Ray70}, the extension~$\L$ must be relatively ample since its generic fiber is ample. 
\end{proof}
We mention an elementary but useful feature of working over a one-dimensional base.
\begin{rem}[flatness]\label{flatness over Dedekind scheme}
	Every non-constant morphism $X\to B$  from an integral scheme $X$  to a Dedekind scheme is flat by  \cite[Corollary 4.3.10]{Liu}. 
\end{rem}

We shall now proceed to the construction of the Néron--Tate height pairing in both the arithmetic and geometric settings. In the former setting, this  requires the notion of admissibility.
\subsection{Admissible metrics}

\begin{defn}\label{definition of admissible metric}
	Let $G$ be a smooth complex commutative algebraic group  and let $L$ be either a symmetric or an antisymmetric  line bundle on $G$. A  continuous metric $\|\cdot\|$ on $L$ is called \textbf{admissible}, if there is an isometry
	\[
	[n]^*(L,\|\cdot\|)\cong n^\epsilon (L,\|\cdot\|)
	\]  some $n\in \Z_{>1}$, where $\epsilon=2$ if $L$ is symmetric, and $\epsilon=1$ if $L$ is antisymmetric.
	
	The definition can be adapted to the relative setting. Let $G/S$ be a  smooth commutative group scheme over a complex  variety $S$. Let $L$ be either a symmetric or an antisymmetric  line bundle on $G$. A  continuous metric $\|\cdot\|$ on $L$ is called \textbf{admissible}, if for every $x\in S(\C)$, the fiber $(L(x),\|\cdot\|_x)$ is an admissible metrized line bundle.
\end{defn}
\begin{void}[admissible metric and rigidification]\label{admissible metric and rigidification}
	If $G$ is proper, an admissible metric on the trivial line bundle~$\O_G$ is a constant function on $G^\an$ (complex analytification of $G$).  We consider $\O_G$ as a line bundle rigidified at the identity element $e\in G(\C)$, i.e.\ we choose an isomorphism $\phi\colon e^*\O_G\xrightarrow{\sim} \C$. Using the rigidification,  we can normalize the  admissible metric $\|\cdot\|$ by requiring $\|\phi^{-1}(1)\|=1$.
\end{void}
\begin{prop}\label{admissible metric on line bundle on abelian scheme}
	If $G/S$ is an abelian scheme and $L$ is  rigidified along the identity, then there is a unique admissible metric on $L$, normalized along the identity section of $G/S$, which is moreover a smooth metric. 
\end{prop}
\begin{proof}
	See \cite[\S III]{MB_adm}. 
\end{proof}
\begin{rem}
	When $G/S$ is an abelian scheme, it can be shown that if $(L,\|\cdot\|)$ is admissible metrized w.r.t.\ some $n>1$, then it is also admissible w.r.t.\ any positive integer $m>1$. A detailed proof can be found in \cite[Lemma 3.4]{Heinz}. If $\dim S=1$, the statement also follows from Theorem \ref{adelic extension summary, analytic case}. 
\end{rem}
\begin{rem}
	When $G$ is an abelian variety, the curvature form $c_1(L,\|\cdot\|)$ of  an admissible metrized line bundle is translation-invariant. In fact, this property characterizes admissible metrics, see \cite[Propostion 3.6]{MB_adm}. 
\end{rem}
\subsection{The construction}

Let $K$ be a number field or the function field of a smooth $k$-curve, where $k$ is  a field.  Write $B=\Spec O_K$ in the number field case, and write $B$ for the smooth projective $k$-curve such that $k(B)\cong K$ in the function field case. Let $A$ be an abelian variety over $K$. Let $P$ be the Poincaré bundle  associated with $A$. Let $\overline{\PP}$ be the extended Poincaré bundle  (Proposition \ref{prolongation of  Poincaré  bundle}). In the number field case, by Proposition \ref{admissible metric on line bundle on abelian scheme}, $\overline{\PP}(\C)$ admits a unique smooth  admissible metric $\|\cdot\|$.

\begin{rem}\label{Rmk: order of component group is finite}
	Since the component groups of  $\NN(A)$ at each  place non-archimedean place $v$ of $K$ are trivial for all but finitely many~$v$, there is a positive integer $n$ such that $nx$ extends to a section $\bar{x}\in \NN^0(A)(B)$ for every $x\in A(K)$.
\end{rem}
Fix an algebraic closure $K\subset \overline{K}$.
\begin{defn}\label{Def: NT height pairing}
	Let $x\in A(\overline{K})$ and $y\in A^\vee(\overline{K})$. Let $n$ be a positive integer such that $nx$ extends to $\bar{x}\in \NN(A)^0(B)$ and $ny$ extends to $\bar{y}\in \NN(A^\vee)(B)$.  Let $K'/K$ be a finite field extension such that $x\in A(K'),\, y\in A(K')$. In the arithmetic case, the \textbf{Néron--Tate height pairing}  $\langle x,y\rangle_A$ of $x, y$ is given by 
	\[
	\langle x,y\rangle_A\coloneq \frac{\widehat{\deg}((\bar{x},\bar{y})^*(\overline{\PP},\|\cdot\|))}{n^2[K':K]}\in \R.
	\] The definition is independent of $n$ and the choice of $K'$. 
	
	If $K\cong k(B)$ is a function field, we use a similar formula to define the Néron--Tate height pairing:
	\[
	\langle x,y\rangle_A\coloneq \frac{{\deg}((\bar{x},\bar{y})^*\overline{\PP})}{n^2[K':K]}\in \Q.
	\] Again, the definition is independent of $n$ and the choice of $K'$. 
	
	The biextension structure of $\overline{\PP}$ and the compatibility of the admissible metric with this biextension structure (arithmetic case), see \cite[\S 5.2]{MB_adm}, give a bilinear form
	\[
	\langle -,- \rangle_A\colon A(\overline{K})\times A^\vee(\overline{K})\to \R,
	\] (the codomain is $\Q$ in the function field case) called the \textbf{Néron--Tate height pairing} on~$A$.
\end{defn}

We discuss the relation between the Néron--Tate height pairing (the way we define it) and the Néron--Tate height function. Recall the latter is defined as follows. 
\begin{defn}\label{def. NT height}
	Let $\theta$ be a symmetric ample line bundle on $A$.  The \textbf{Néron--Tate height} (function)  $\hat{h}_\theta\colon A(\overline{K})\to \R$ is given by
	\[
	\hat{h}_\theta(x)\coloneq \lim\limits_{n\to \infty}\frac{h_\theta([2^n]x)}{4^n}\in \R,
	\] where $h_\theta$ is any Weil height associated to $\theta$.  (The limit indeed exists, \emph{cf}.\ \cite[Lemma 9.2.4]{BG06}.)
\end{defn}
Consider the symmetric bilinear pairing
\begin{equation}\label{bilinear form associated to NT height}
	b\colon A(\overline{K})\times A(\overline{K})\to \R,\quad (x,y)\mapsto \hat{h}_\theta(x+y)-\hat{h}_\theta(x)-\hat{h}_\theta(y).
\end{equation}
Let 
\[
\phi_\theta\colon A\to A^\vee,\quad x\mapsto x^\vee\coloneq \tau_x^*\theta-\theta
\] be the polarization (isogeny) induced by $\theta$, where $\tau_x$ is the translation-by-$x$ automorphism. Then we have
\begin{equation}\label{2times Neron-Tate}
	2\hat{h}_\theta(x)=b(x,x)=\langle x,x^\vee\rangle_A,
\end{equation} where the second equality is by  \cite[Proposition 9.3.6]{BG06}.

\section{Adelic line bundles}\label{The Notion of Adelic Line Bundles}
We briefly review  Yuan--Zhang's theory \cite{YZ21}  of adelic line bundles and fix our notation at the same time. 

In this section, $k$ denotes either the archimedean-valued field $\C$ (complex-analytic case) or the ring $\Z$ of integers (arithmetic case).
\begin{defn}\label{projective model, adelic setting}
	Let $\U$ be an integral scheme that is quasi-projective and flat over $k$. A \textbf{projective model}~$\X$ of $\U$ is an integral projective $k$-scheme  that contains $\U$ as a Zariski open subset.
\end{defn} Projective models of $\U$ form a directed system with transition maps induced by birational morphisms that are isomorphic over $\U$.

For an integral  scheme $\X$ that is projective flat over $k$, we have the group $\overline{\Pic}(\X)$  of  isometry classes of \emph{continuous} metrized line bundles on the complex-analyti\-fication $\X^\an=\X(\C)$. The concept of metrized divisors are modelled by Green's functions as we now explain. Let $X$ be a complex variety and $X^\an$  the associated complex analytic space.
\begin{defn}
	Let $D$ be a  divisor\footnote{By default, a divisor is meant to be a Cartier divisor.} on  $X$. A \textbf{continuous Green's function}  $g_D$ for $D$ (or for the analytified divisor $D^\an$) is a  continuous function
	\[
	g_D\colon X^\an\ohne |D^\an|\to \R
	\] such that for any local equation $f$ for $D$ such that $\divisor(f)=D|_U$ for some Zariski open subset $U\subset X$, the function $g_D+\log|f|$ extends to a continuous function on $U^\an$. 
\end{defn}
\begin{rem}
	One can define \emph{smooth} Green's function in a similar manner, where, if $X$ is not smooth (or if $X^\an$ is not a complex manifold), the smoothness is defined as in \cite[\S 2.1.1]{YZ21}.
\end{rem}
\begin{void}\label{Green's function and line bundle}
	A continuous Green's function $g_D$ induces a continuous metric on the analytic line  bundle $\O_{X^\an}(D)$ by the rule
	\[
	\|1_D\|=\exp(-g_D),
	\] where $1_D$ denotes the canonical rational section of $\O_{X}(D)$ associated to $D$.
	
	Conversely, given a continuous metrized  line bundle $(L,\|\cdot\|)$ on $X$ and a nonzero rational section $s$ of $L$, we get a continuous Green's function $g_D$ for $D\coloneq\divisor(s)$ by the rule
	\[
	g_D\coloneq -\log\|s\|.
	\] The two constructions are inverse to each other.
\end{void}
\begin{defn}
	A \textbf{continuous metrized divisor} on $X$ is a pair $(D,g_D)$, where $D$ is a  divisor on $X$ and $g_D$ is a continuous Green's function for $D$. Let $\overline{\Div}(X)$ be the group of continuous metrized divisors on $X$. A \textbf{principal metrized divisor} is a metrized divisor of the form $(\divisor(f),-\log|f|)$, where $f$ is a  rational function on $X$. By an \textbf{effective} metrized divisor $(D,g_D)$, we mean  a metrized divisor such that $D$ is an effective (Cartier) divisor and $g_D\geq 0$. We define
	\[
	(D,g_D)\leq (D',g_{D'})  \iff (D-D',g_{D}-g_{D'}) \text{ is effective}.
	\]
\end{defn}
\begin{rem}
	By \ref{Green's function and line bundle}, the canonical map $\Div(X)\to \Pic(X),\, D\mapsto \O_X(D)$
	admits a metric refinement $\overline{\Div}(X)\to \overline{\Pic}(X)$. If $X$ is smooth projective, then this map is surjective and has kernel exactly the subgroup of principal metrized divisors. This follows from the Poincaré--Lelong formula,  see e.g.\ \cite[\S II.1]{Sou+}. 
\end{rem}
Let $\U$ be an integral scheme that is quasi-projective flat over $k$. 

\begin{defn}\label{analytic adelic divisors on quasi-projective complex variety}
	Projective models of $\U$ form a directed system with transition maps induced by birational morphisms that are isomorphic over $\U$. We let
	\[
	\overline{\Div}(\X,\U)=\overline{\Div}(\X)_\Q\times_{\Div(\U)_\Q} \Div(\U)
	\]
	be the group of  $(\Q,\Z)$-\textbf{divisors}, formed by the fiber product of abelian groups along the natural maps. Let
	\[
	\overline{\Div}(\U/k)_{\mod}=  \colim_{\X} \overline{\Div}(\X,\U)
	\]
	be the group of \textbf{model  adelic divisors}, where the filtered colimit ranges over    projective  models $\X$ of $\U$, and the transition maps are induced by pullback of metrized divisors along birational morphisms that are isomorphisms over $\U$.
\end{defn}
\begin{defn}
	Choose a projective model $\X\supset \U$.  A \textbf{metrized boundary divisor} (w.r.t.\ $\X$) is a pair $\overline{\scrB}=(\scrB,g_B)$, where $\scrB$ is an effective  divisor on $\X$ such that $\X\ohne |\scrB|=\U$, and $g_B$ is a \emph{strictly positive} continuous Green's function for the analytified divisor $B\coloneq \scrB^\an$. We call  $\scrB$ a \textbf{boundary divisor} of $\U$ on $\X$.
\end{defn}
\begin{void}[existence of metrized boundary divisors]\label{existence of metrized boundary divisor}
	Upon blowing up $\X$ along $\X\ohne \U$, we can find a boundary divisor of $\U$ on some projective model of $\U$. By  the construction outlined in \ref{Green's function and line bundle}, to show that a divisor admits a continuous Green's function, it is equivalent to construct a continuous metric on the associated line bundle, but this can be done by a partition of unity argument, which only requires para-compactness. The complex analytic space associated to a complex variety is para-compact. Moreover, a Green's function for an effective divisor is bounded below, hence, by adding a sufficiently large positive constant, we may arrange that  the Green's function for a boundary divisor is strictly positive.
\end{void}	
\begin{defp}
	Let $\overline{\scrB}=(\scrB,g_\scrB)$ be a metrized boundary divisor of~$\U$ on some projective  model $\X$. The \textbf{boundary norm} w.r.t.\ $\overline{\scrB}$ on $\overline{\Div}(\U/k)_\mod$ is defined to be
	\[
	\|\overline{\scrD}\|_{\overline{\scrB}}\coloneq \inf\{\epsilon\in \Q_{>0}\mid -\epsilon\overline{\scrB}\leq \overline{\scrD}\leq \epsilon \overline{\scrB}\}\in \R_{\geq 0}\cup\{\infty\}.
	\] Here the order ``$\leq$'' is induced by the effectivity of metrized  divisors, and  $\pm\epsilon\overline{\scrB}, \overline{\scrD}$ are compared implicitly on a common  projective  model of $\U$.  This defines an extended norm on $\overline{\Div}(\U/k)_\mod$,  and it is independent of the choice of  metrized boundary divisor.
\end{defp}
\begin{proof}
	See \cite[Lemma 2.4.1]{YZ21}; in the complex-geometric case see also  \cite[Propostion 3.2]{Song}.
\end{proof}
\begin{rem}
	Note that $\|\overline{\scrD}\|_{\overline{\scrB}}=\infty$ if and only if $\scrD|_\U\neq 0$.
\end{rem}

\begin{defn}\label{Def: Cauchy seq}
	We form  the group $\widehat{\Div}(\U/k)$ of \textbf{adelic divisors} on $\U$ by completing $\overline{\Div}(\U/\C)_\mod$ w.r.t.\ a boundary norm $\|\cdot\|_{\overline{\scrB}}$.  
	
	By design, an  adelic divisor is represented by a \textbf{Cauchy sequence} $(\overline{\scrD}_i)_{i\geq 1}$ of model  adelic divisors, that is, there is a sequence $(\epsilon_i)_i$ of positive rational numbers converging to~0 such that
	\[
	\forall\, j\geq i\geq 1\colon -\epsilon_i\overline{\scrB}\leq \overline{\scrD}_j-\overline{\scrD}_i\leq \epsilon_i\overline{\scrB} 
	\] for some metrized boundary divisor $\overline{\scrB}$.
\end{defn}
\begin{void}
	Let $\L$ be a line bundle on $\U$. Let $\X_1, \X_2$ be two projective models of $\U$.   Let $\LL_i$ be  $\Q$-line bundles on  $\X_i$ $(i=1,2)$  that extend $\L$, that is,   we are given isomorphisms  
	$\LL_1|_\U\xrightarrow{\sim} \L$ and $\LL_2|_\U\xrightarrow{\sim} \L$ of $\Q$-line bundles. We denote by $f\colon \LL_1\dashrightarrow \LL_2$ for the unique ``\textbf{rational map}'' of $\Q$-line bundles induced by the isomorphisms $\LL_1|_\U\xrightarrow{\sim} \L\xrightarrow{\sim} \LL_2|_\U$. We have the obvious definition for the rational map $f^{-1}\colon \LL_2\dashrightarrow \LL_1$.
	
	Suppose we have a birational morphism $p\colon \X_2\to \X_1$ which is an isomorphism over $\U$. Let $\LL_i$ be $\Q$-line bundles on $\X_i$. A rational map $f\colon \LL_1\dashrightarrow\LL_2$ induces a $\Q$-divisor 	$\divisor_{\X_2}(f)$ supported on $\X_2\ohne \U$ by viewing $f$ as a rational section of $\LL_2-p^*\LL_1$, which is nowhere vanishing on $p^{-1}\U$. Then $\divisor_{\X_2}(f)$ gives rise to a model adelic divisor on $\U$, denoted $\divisor(f)$; as the notation suggests, it is independent of the choice of models, because projective models of $\U$ form a filtered system. 
\end{void}

\begin{defn}
	Fix a metrized boundary divisor $\overline{\scrB}$  on a projective  model $\X$ of~$\U$. Consider a tuple $(\L,(\X_i,\overline{\LL_i},\ell_i)_{i\geq 1})$ where
	\begin{itemize}
		\item  $\L$ is a  line bundle on $\U$;
		\item Each $\X_i$ is a  projective  model of $\U$;
		\item Each $\LL_i$ is a   $\Q$-line bundle on $\X_i$;
		\item For every $i\geq 1$, an isomorphism
		\[
		\ell_i\colon \L\xrightarrow{\sim}\LL_i|_\U
		\] of $\Q$-line bundles;
		\item For every $i\geq 1$, $\LL_i$ is endowed with a continuous metric $\|\cdot\|_i$ such that  the rational map  $\ell_i\ell_1^{-1}\colon \LL_1\dashrightarrow \LL_i$ induces an isometry
		\[
		(\LL_1|_\U,\|\cdot\|_1)\xrightarrow{\sim} (\LL_i|_\U,\|\cdot\|_i);
		\] write $\overline{\LL_i}=(\LL_i,\|\cdot\|_i)$.
	\end{itemize}	
	The tuple $(\L,(\X_i,\overline{\LL_i},\ell_i)_{i\geq 1})$ represents an \textbf{adelic line bundle} on $\U$ if $\divisor(\ell_i\ell_1^{-1})_{i\geq 1}$ represents an adelic divisor on $\U$. Explicitly, this means that there is  a sequence $(\epsilon_i)_i$ of positive rational numbers converging to 0 such that the \textbf{Cauchy condition}
	\[
	-\epsilon_i\overline{\scrB}\leq \overline{\divisor}(\ell_j\ell_i^{-1})\leq \epsilon_i\overline{\scrB},\quad j\geq i\geq 1
	\] is satisfied, where $\overline{\divisor}(\ell_j\ell_i^{-1})\coloneq(\divisor(\ell_j\ell_i^{-1}),g_\scrD)$ is a metrized divisor, i.e.\ $g_\scrD$ is Green's function for $\divisor(\ell_j\ell_i^{-1})$ (realized on a suitable  projective  model of $\U$). Informally, we say that adelic line bundles are limits of model  adelic line bundles. 
\end{defn}

\begin{defn}[{\cite[\S 2.5.1]{YZ21}}]\label{iso of adelic line bundles, arch. case} Fix the boundary data $(\X,\overline{\scrB})$.
	Let $(\L, (\X_i,\overline{\LL_i},\ell_i)_{i\geq 1})$ and $(\L', (\X'_i,\overline{\LL'_i},\ell'_i)_{i\geq 1})$ be two tuples representing two adelic line bundles on $\U$. A \textbf{morphism} between the tuples  is an isomorphism $\iota\colon \L\xrightarrow{\sim}\L'$ of (integral) line bundles satisfying the following properties. For each $i\ge 1$, the rational map 
	\[
	\ell_i'\iota\ell_i^{-1}\colon \LL_i\dashrightarrow \LL'_i	\] induces a model adelic divisor $\divisor(\ell'_i\iota\ell_i^{-1})$. Choosing a Green's function for this divisor, we obtain a model adelic divisor $\overline{\divisor}(\ell'_i\iota\ell_i^{-1})\in \overline{\Div}(\U/k)_{\mathrm{mod}}$ (whose image in $\Div(\U)$ is~0). And we require the sequence 
	$(\overline{\divisor}(\ell'_i\iota\ell_i^{-1}))_{i\geq 1}$  to be Cauchy.
	
	We let $\widehat{\Pic}(\U/k)$ be the  group  of isomorphism classes of such triples and call elements thereof also as \textbf{adelic line bundles} on $\U$ (the group structure is induced by the tensor product of metrized line bundles). Note that the identity element is represented by $(\O_\U,(\X,\overline{\O}_\X,1))$, where the metrized line bundle $\overline{\O}_\X$ corresponds to a choice of a continuous function on~$\X^\an$.
\end{defn}
The arithmetic setting ($k=\Z$) enters in the following way.  Suppose we have a smooth quasi-projective variety $S$ over a number field $K$, and $X/S$ is a smooth (quasi-)projective family. Then the composition $X\to S\to \Spec K\to \Spec O_K$ is not quasi-projective (since it is not of finite type), where $O_K$ denotes the ring of algebraic integers of $K$. On the other hand, $X$ is essentially quasi-projective over $O_K$,  see \cite[\S 2.3.1]{YZ21}.  We also regard~$X$ as an essentially quasi-projective scheme over $\Z$.
(We do not need to introduce the essential quasi-projectivity  in the complex-geometric setting; every quasi-projective $\C$-scheme is essentially quasi-projective over $\C$.)

For an essentially quasi-projective integral flat $O_K$-scheme $X$, one still has the notion of adelic line bundles.
\begin{defn}\label{quasi-projective model, arithmetic}
A \textbf{quasi-projective model} of~$X$ is an integral scheme $\U$ that is quasi-projective flat over an open subset $V\subset \Spec O_K$, together with a $K$-open immersion $X\hookrightarrow \U\times_V\Spec K$.
\end{defn}
Quasi-projective models of $X$ form a directed system with transition maps induced by birational morphisms that are isomorphisms over $X$.
\begin{defn}\label{arithmetic adelic divisors and line bundles on quasi-projective $K$-varieties}
We define
	\[
	\widehat{\Div}(X/\Z)=\colim_\U\widehat{\Div}(\U/\Z)
	\] as the group of \textbf{adelic divisors} on $X$, where the filtered colimit is taken over all quasi-projective models $\U$ of $X$.  We define
	\[
	\widehat{\Pic}(X/\Z)=\colim_\U\widehat{\Pic}(\U/\Z)
	\] as the group of (isomorphism classes of) \textbf{adelic line bundles} on $X$.
\end{defn}
\begin{void}[pullbacks]\label{pullback of adelic line bundles}
	Recall $k\in  \{\C,\Z\}$. Let $X, Y$ be essentially quasi-projective integral flat $k$-schemes, and suppose there is a $k$-morphism $f\colon X\to Y$. By  \cite[\S 2.5.5]{YZ21}, there is a well-defined pullback
	\[
	f^*\colon \widehat{\Pic}(Y/k)\to \widehat{\Pic}(X/k)
	\]   of adelic line bundles.
\end{void}

To distinguish, we sometimes call adelic line bundles (resp.\ divisors) in the complex-analytic case \textbf{analytic}  adelic line bundles (resp.\ divisors), and those in the arithmetic case \textbf{arithmetic} adelic line bundles (divisors).

\section{Complex-analytic case}\label{Sect:Complex-analytic Case}
\subsection{Review of positivity notions}
We first need to identify an initial input $(\X,\overline{\LL})$, where the metrized line bundle $\overline{\LL}$  satisfies certain positivity properties. 
\begin{void}[nefness and integrability]\label{nefness and integrability}
	Recall that a line bundle $L$ on a  variety $X$ over a field $k$ is called \textbf{nef} if, for every smooth proper  $k$-curve $C\subset X$, $\deg(L|_C)\geq 0$.
	
	Now let $X$ be a  projective complex variety and $L$  a line bundle on $X$. A continuous metric  $\|\cdot\|$ on $L$ is called \textbf{semipositive} or \textbf{psh} (plurisubharmonic) if the first Chern current associated to $(L,\|\cdot\|)$ is a positive $(1,1)$-current. The  metric $\|\cdot\|$ is called \textbf{integrable} if it is a difference of two semipositive metrics. A continuous metrized line bundle $(L,\|\cdot\|)$ is called \textbf{nef} if $L$ is nef and  $\|\cdot\|$ is semipositive. An \textbf{integrable metrized line bundle} is a difference of two nef metrized line bundles. See also the discussion at the end of \cite[\S 2.1.1]{YZ21} on different conventions of these notions.
\end{void}
\begin{ex}\label{admissible line bundle is integrable}
	Let $(L,\|\cdot\|)$ be an admissible line bundle on a complex abelian variety~$X$ (Definition \ref{definition of admissible metric}). By work of Zhang \cite{ZhangSmallPoints}, it is a smooth integrable metrized line bundle.
\end{ex}
The following properties for ample line bundles are well-known for a proper family. We recall these properties in the not necessarily proper setting.
\begin{lem}\label{GW, 13.62}
	Let $f\colon X\to S$ be a morphism of finite type to a quasi-compact scheme $S$. The following assertions for a line bundle $L$ on $X$ are equivalent:
	\begin{enumerate}
		\item $L$ is $f$-ample (relatively ample w.r.t.\ $f$).
		\item There exists a positive integer $n$ such that $nL$ is $f$-very ample.
		\item For every line bundle $L'$ on $X$ there exists a positive integer $n_0$ such that $L'+nL$ is $f$-very ample for all $n\geq n_0$. 
	\end{enumerate}
\end{lem}
\begin{proof}
	See \cite[Theorem 13.62]{GW}.
\end{proof}
\begin{rem}
	Let $f\colon X\to R$ be a morphism of finite type to an \emph{affine} scheme $R$, then an $f$-ample line bundle on $X$ is  (absolutely) ample, see \cite[Proposition 13.63]{GW}.
\end{rem}
\begin{lem}\label{GW, 13.65}
	Let $R$ be an affine scheme. Let $f\colon X\to S$ be a morphism of finite type over an $R$-scheme $S$. Let $L$ be an $f$-ample line bundle on $X$, and let $M$ be an (absolutely) ample line bundle on $S$. There exists a positive integer $n_0$ such that $L+f^*(nM)$ is (absolutely) ample for all $n\geq n_0$.
\end{lem}
\begin{proof}
	This is a special case of \cite[Proposition 13.65]{GW}.
\end{proof}
\begin{lem}\label{extension of line bundle, positivity}
	Let $k$ be a field and $S$ be a smooth $k$-curve; let  $\widetilde{S}\supset S$ be the smooth projective model over $k$.  Let $p\colon \U\to S$ be a smooth quasi-projective morphism. Let $\L$ be an $S$-\emph{ample} line bundle on $\U$. There exist a projective flat morphism $\tilde{p}\colon \X\to \widetilde{S}$, where~$\X$ is an integral projective $k$-scheme containing $\U$ as a dense  open subset, and a $\Q$-line bundle~$\LL$ on $\X$ with the following properties:
	\begin{enumerate}
		\item the family $\X/\widetilde{S}$ extends $\U/S$, that is, there is a commutative (but not necessarily a pullback) diagram
		\[
		\begin{tikzcd}
			\U\ar[hook]{r} \ar{d}{p} & \X \ar{d}{\tilde{p}} \\
			S\ar[hook]{r} & \widetilde{S}
		\end{tikzcd}
		\]
		\item there exists an isomorphism $\LL|_\U\xrightarrow{\sim} \L$ of $\Q$-line bundles.
		\item there exists an ample line bundle $\mathscr{M}$ on $\widetilde{S}$ such that $\LL+\tilde{p}^*\mathscr{M}$ is (absolutely) ample on~$\X$.
	\end{enumerate} 
\end{lem}
\begin{proof}
	This is  similar to the proof in \cite[Theorem 6.1.1]{YZ21} (see the  last few paragraphs). 
	
	Since $\L$ is relatively ample, by Lemma \ref{GW, 13.65}, we can choose an ample line bundle $\widetilde{M}$ on~$\widetilde{S}$ such that	 some positive multiple of $\L+p^*M$, where $M\coloneq \widetilde{M}|_S$, induces a locally closed immersion of $\U\hookrightarrow \P^n_k$ into some projective space over $k$; set $\X$ to be the scheme-theoretic closure of $\U$ in~$\P^n_k$. 
	
	By construction, we  obtain a rational map $\X\dashrightarrow \widetilde{S}$. Upon blowing up $\X$ further at centres not meeting $\U$, we can turn it into an (actual) projective morphism $\tilde{p}\colon \X\to \widetilde{S}$. Note that $\tilde{p}$ is flat by Remark \ref{flatness over Dedekind scheme}. By construction, $\tilde{p}$ extends $p$.
	
	By design, a suitable root $\LL'$ of the tautological line bundle $\O(1)$ on $\P^n_k$ restricted along the closed subvariety $\X\subset \P^n_k$ extends  $\L+p^*{M}$ and it is absolutely ample on $\X$. We are done by setting $\LL\coloneq \LL'-\tilde{p}^*\widetilde{M}$.
\end{proof}

Let $\U/S$ be as in Lemma \ref{extension of line bundle, positivity} with $k=\C$, and let $\X/\widetilde{S},\, \LL,\, \mathscr{M}$ be  the corresponding extension data. Equip $\LL$ with a smooth metric  and denote the resulting metrized line bundle by $\overline{\LL}$. By ampleness, we can equip $\mathscr{M}$ with a  semipositive metric such that
$\overline{\LL}+\tilde{p}^*\overline{\mathscr{M}}$ is a nef metrized line bundle on~$\X$.

We thus obtain the following metric refinement of Lemma \ref{extension of line bundle, positivity}.
\begin{lem}\label{extension of line bundle, positivity, anayltic}
	Let $p\colon \U\to S$ be a smooth quasi-projective morphism over a smooth $\C$-curve $S$. Let $\L$ be an $S$-\emph{ample} line bundle on $\U$. There exist a projective flat morphism $\tilde{p}\colon \X\to \widetilde{S}$ from a projective  model $\X\supset \U$  and a metrized $\Q$-line bundle $\overline{\LL}$ with the following properties:
	\begin{enumerate}
		\item the family $\X/\widetilde{S}$ extends $\U/S$;
		\item there exists an isomorphism $\LL|_\U\xrightarrow{\sim} \L$ of $\Q$-line bundles.
		\item there exists a nef metrized line bundle $\overline{\mathscr{M}}$ on $\widetilde{S}$ such that the metrized line bundle $\overline{\LL}+\tilde{p}^*\overline{\mathscr{M}}$ on $\X$ is nef.
	\end{enumerate} 
\end{lem}
\begin{rem}
	By restriction, we obtain similar extension results as in Lemma \ref{extension of line bundle, positivity, anayltic} over any smooth partial compactification $S\subset \Sbar$.
\end{rem}
\subsection{Constructing models of $(\NN(A)^0,\L)$}
The construction relies on  the  procedure of resolving the indeterminacy locus of a morphism. For completeness, we explain this in a general setting.
\begin{void}[resolving indeterminacy loci]\label{Resolve indeterminancy locus}
	Let 	$\U\subset \X$  be an dense open subscheme of a noetherian separated integral scheme $\X$. Let $\Sbar$ be an integral noetherian separated scheme. Suppose $\X$ admits a finite type morphism to $\Sbar$ (then $\X\to \Sbar$ is separated). Suppose there is an $\Sbar$-morphism $f\colon\U\to \U$, regarded as a rational map $f\colon \X \dashrightarrow \X$.
	
	Let $\X_1=\overline{\Gamma_f}\subset \X\times_{\Sbar} \X$ be the scheme-theoretic closure  of the graph $\Gamma_f\subset \U\times_{\Sbar}\U$. Let $p_1=\pr_1|_{\X_1}\colon \X_1\to \X$ be the morphism induced by the projection $\pr_1\colon \X\times_{\Sbar} \X\to \X$ onto the first factor. We thus obtain a  diagram
	\begin{equation}
		\begin{tikzcd}
			\X\ar[dashed]{r}{f} & \X\\
			\X_1 \ar{u}{p_1} \ar[swap]{ur}{f_1}
		\end{tikzcd}
	\end{equation}
	where $f_1$ and $p_1$ are (honest)  morphisms.  Note that if $\X/\Sbar$ is projective, then $f_1$ and $p_1$ are projective as well.
\end{void}
\begin{lem}\label{identified open}
	One has an identification
	\[
	p_1^{-1}\U\cong \U.
	\]  It follows that $p_1$ is birational.
\end{lem}
\begin{proof}
	The identification follows from a chain of isomorphisms
	\[
	p_1^{-1}\U\cong \U\times_{\X} \X_1 \cong \X_1\cap (\U\times_{{\Sbar}} \X)\cong\Gamma_f\cong \U,
	\] where the intersection takes place in $\X\times_{\Sbar} \X$. Except for the second last isomorphism, all the others follow from the definition. To see the remaining one, we specialize \cite[\href{https://stacks.math.columbia.edu/tag/081I}{Tag 081I}]{SP} to the following situation; the left hand side denotes the notation used in the Stacks Projects: $(f\colon X\to Y)=(\U\times_{\Sbar} \X \hookrightarrow \X\times_{{\Sbar}} \X)$ (which is flat); $(g\colon V\to Y)=(\Gamma_f\hookrightarrow \X\times_{{\Sbar}} \X)$ (which is quasi-compact). The $Z'$ in loc.\ cit.\ is just our $\Gamma_f$, since $\Gamma_f\subset \U\times_{\Sbar} \U$ is a closed immersion by separatedness of $f$. To conclude, note that the $Z$ in loc.\ cit.\ is exactly our~$\X_1$.
\end{proof}

Now we apply  Lemma \ref{extension of line bundle, positivity, anayltic} to the following situation. Let $A/S$ be an abelian scheme over a smooth complex curve $S$. Let $S\subset \Sbar$ be a smooth partial compactification. Let $L$ be a symmetric $S$-ample  line bundle on $A$, rigidified along the identity section of $A/S$. By Lemma \ref{ample extension}, it extends uniquely to an $\Sbar$-ample line bundle  $\L$ on $\NN(A)^0$, rigidified along the identity section of $\NN(A)^0/\Sbar$. Write $\U=\NN(A)^0$.

Let $(\X,\overline{\LL})$ be a pair  as in Lemma \ref{extension of line bundle, positivity, anayltic} extending $(\U,\L)$. We then specialize \ref{Resolve indeterminancy locus} to the inclusion $\U\subset \X$ and to the morphism $f\coloneq[2]\colon \U\to \U$, the multiplication-by-2 map over~$\Sbar$.

\begin{conv}
	In the rest of this section we assume our projective models to be \emph{normal}. 
\end{conv}
Note that $\U=\NN(A)^0$ is a regular scheme, in particular normal. Thus, the above normality assumption is harmless, since passing to the normalization of a general projective model of~$\U$ does not alter the normal scheme $\NN(A)^0$. Note also that normal projective models form a directed system via birational morphisms that are isomorphisms over $\U$.

\begin{void}[sequence of projective  models]\label{Sequence of projective models, geometric case}
	Write $\X_0=\X$. For $i\geq 1$  let 
	\[
	\X_i=\overline{\Gamma_{f^i}}\subset \X_{i-1}\times_{\Sbar}\X_{i-1},
	\] be the scheme-theoretic closure of the $i$-th iteration of $f$. By Lemma \ref{identified open}, we have an identified open subscheme  $\U\hookrightarrow \X_i$ for all $i$. We can visualize the situation as follows:
	\begin{equation}\label{Sequence of projective models, visualize}
		\begin{tikzcd}[column sep=large, row sep=large]
			&	\U\ar{r}{f}\ar[hook]{d} \ar[hook, bend right=30]{dd} \ar[hook, bend right=40]{ddd} 
			& \U \ar{r}{f} \ar[hook]{d} 
			& \U\ar[hook]{d} \\
			& \X_0\ar[dashed]{r}{f} 
			& \X_0\ar[dashed]{r}{f} 
			& \X_0 \\
			&	\X_1\ar[dashed]{r}{f} \ar[swap]{u}{p_1} \ar{ur}{f_1}
			& \X_1 \ar[swap]{u}{p_1} \ar{ur}{f_1}
			& \ddots \\
			&	\X_2\ar[swap]{u}{p_2} \ar{ur}{f_2}  & & 
		\end{tikzcd}
	\end{equation}
	Moreover, the solid arrows commute:
	\begin{equation}\label{commutativity of resolving indeterminancy}
		f_i\circ p_{i+1}=p_i\circ f_{i+1}\colon \X_{i+1}\to \X_n.
	\end{equation} This is because they commute over the identified dense open subset $\U$, and the $\X_i$'s are reduced and separated. 
\end{void}
\begin{void}[sequence of model line bundles]\label{Sequence of model line bundles}
	Recall that $\L$ is symmetric and $f=[2]$. We have a unique isomorphism 
	\[
	f^*(\LL|_\U)\xrightarrow{\sim} 4\LL|_\U,
	\]   of rigidified line bundles (by Lemma \ref{equivalence of cat. of rigidified line bundle} and the theorem of the cube for $L\in \Pic(A)$).
	
	Put $\LL_0\coloneq \LL$ and for $i\geq 1$, we define
	\[
	\LL_i\coloneq \frac{1}{4}f_i^*\LL_{i-1}\in \Pic(\X_i)_\Q.
	\] We define    the morphisms $g_i\colon \X_i\to \X_1$ ($i\geq 1$) inductively by
	\begin{displaymath}\label{g_i of seq. of projective models}
		g_1=\id_{\X_1}, ~ g_2=f_2, ~ g_3= g_2\circ f_3, ~ g_4=g_3 \circ f_4
	\end{displaymath}
	and so on. Then we find
	\begin{displaymath}
		\LL_i=\frac{1}{4^{i-1}}g_i^*\LL_1.
	\end{displaymath} It is clear that for $i, j\geq 0$, we have
	\begin{equation}\label{L_i equals L_j when restricted to the open}
		\LL_i|_\U\cong \LL_j|_\U.
	\end{equation}
\end{void} 
\begin{lem}\label{lemma for Cauchy sequence, geometric}
	For $i\geq 0$ consider the rational maps
	\[
	\ell_i\colon \LL_0\dashrightarrow \LL_i
	\] induced by \eqref{L_i equals L_j when restricted to the open}, where $\ell_0=\id$. Consider also the rational maps
	\[
	\ell_{i+1}\ell_i^{-1}\colon \LL_i \dashrightarrow \LL_0 \dashrightarrow \LL_{i+1}.
	\] Then we have an equality of $\Q$-divisors on $\X_{i+1}$:
	\begin{equation}\label{equality of model divisor}
		\divisor_{\X_{i+1}}(\ell_{i+1}\ell_i^{-1})=\frac{1}{4^i}g_{i+1}^*\divisor_{\X_1}(\ell_1).
	\end{equation}
	Consequently, we have an equality of model adelic divisors
	\begin{equation}\label{equality of model adelic divisors, geometric case}
		\divisor(\ell_{i+1}\ell_i^{-1})=\frac{1}{4^i}g_{i+1}^*\divisor(\ell_1).
	\end{equation} 
\end{lem}
\begin{proof}
	We prove the lemma by induction on $i\geq 0$. The initial case is trivial. Now suppose \eqref{equality of model divisor} holds for all $i\leq n-1$ ($n\in \Z_{>0}$).  By the induction hypothesis, we are reduced to prove
	\begin{equation}\label{Cauchy sequence, induction step}
		\divisor_{\X_{n+1}}(\ell_{n+1}\ell_n^{-1})=\frac{1}{4}f_{n+1}^*\divisor_{\X_n}(\ell_n\ell_{n-1}^{-1}).
	\end{equation} To see this, we note
	\[
	f_{n+1}^*(\LL_n-p_n^*\LL_{n-1})\cong f_{n+1}^*\LL_n-f_{n+1}^*p_n^*\LL_{n-1}=4(\LL_{n+1}-p_{n+1}^*\LL_n)
	\] by the definition of $\LL_i$ and by the commutativity \eqref{commutativity of resolving indeterminancy}. This concludes the proof.
\end{proof}

Next, we  introduce suitable metrics on the models $\LL_i$ for $i\geq 1$.
\begin{void}\label{metrized model line bundle}
	Let $\|\cdot\|_0$ be any \emph{smooth} metric on $\LL_0$. Recall that the $f_i$'s are obtained by the procedure in \ref{Resolve indeterminancy locus}, which are morally multiplication-by-2 maps.  For $i\geq 1$, we   construct  recursively a metric $\|\cdot\|_i$ on $\LL_i$ such that
	\begin{equation}\label{Tate's limiting trick, local archimedean case}
		(\LL_i,\|\cdot\|_i)\cong \frac{1}{4}f_i^*\left(\LL_{i-1},\|\cdot\|_{i-1}\right)\in \overline{\Pic}(\X_i)_\Q
	\end{equation}
	is an isometry. Write $\overline{\LL_i}\coloneq (\LL_i,\|\cdot\|_i)$ for the  metrized line bundles. 
\end{void}

\begin{prop}\label{adelic extension, geometric case}
	The tuple $(\L, (\X_i,\overline{\LL_i},\ell_i)_{i\geq 1})$ represents an analytic adelic line bundle $\widehat{\L}\in \widehat{\Pic}(\NN(A)^0/\C)$ such that $\widehat{\L}|_A\cong \widehat{L}$.
\end{prop}
\begin{proof}
	Since $\LL_1|_\U\xrightarrow{\sim} \L$, the $\Q$-divisor $\divisor_{\X_1}(\ell_1)$ is supported on the complement of $\U$. It follows that  there is an $r\in \Q_{>0}$ such that the model adelic divisor satisfies
	\begin{equation}\label{inequality of geometric divisor, 1}
		-r\scrB\leq \divisor(\ell_1)\leq r\scrB
	\end{equation}
	(on some birational modification of $\X$). By \eqref{equality of model adelic divisors, geometric case}, we conclude
	\begin{equation}\label{inequality of geometric divisor, 2}
		-\frac{r}{4^i}\scrB\leq \divisor(\ell_{i+1}\ell_i^{-1})\leq \frac{r}{4^i}\scrB.
	\end{equation}
	We want to show that the similar inequalities  hold for the metrized divisors. Let $D=\divisor_\X(\ell_1)$, and let  $g_D$ be a Green's function for $D$. Recall that $g_\scrB$ is a strictly positive Green's function for the boundary divisor $\scrB$ on the  projective model $\X=\X_0$ of $\U$. We claim that there is a constant $c>0$ such that
	\begin{equation}\label{inequality of Green's function}
		|g_D|\leq rg_{\scrB}+c\quad \text{on }\U^\an.
	\end{equation}
	In fact, by compactness of $\X^\an$, it suffices to shows the inequality locally on a neighbourhood of $x\in |\scrB|\subset \X^\an$. By the defining property of Green's functions, we may  assume that 
	\[
	g_D=\sum_{j=1}^n a_j\log|t_j|+O(1),\quad g_\scrB=\sum_{j=1}^n b_j\log|t_j|+O(1),
	\] where $a_j,\, b_j\in \Q_{>0}$ and the $t_j$'s are  suitable local coordinates around $x$ such that $D$ is given by the vanishing $t_1^{a_1}=0,\ldots,t_n^{a_n}=0$, and $\scrB$ is given by the vanishing $t_1^{b_1}=0,\ldots,t_n^{b_n}=0$. Then \eqref{inequality of geometric divisor, 1} implies $a_j\leq rb_j$ for all $j$. Hence \eqref{inequality of Green's function} holds locally around $x$. 
	
	On the other hand, since $g_\scrB$ is strictly positive, there exists $b>0$ such that $g_\scrB\geq b$. Then, for $r'\coloneq r+c/b>r>0$, we have
	\[
	-r'(\scrB,g_\scrB)\leq (D,g_D)\leq r'(\scrB,g_\scrB).
	\] We thus obtain
	\[
	-r'\overline{\scrB}\leq \overline{\divisor}(\ell_1)\leq r'\overline{\scrB}.
	\] And hence, by \eqref{inequality of geometric divisor, 2}, we conclude 
	\[
	-\frac{r'}{4^i}\overline{\scrB}\leq \overline{\divisor}(\ell_{i+1}\ell_i^{-1})\leq \frac{r'}{4^i}\overline{\scrB}
	\] so that $(\overline{\divisor}(\ell_i))_{i\geq 1}$ is a Cauchy sequence w.r.t.\ boundary norm induced by $\overline{\scrB}$. Therefore, the tuple $(\L, (\X_i,\overline{\LL_i},\ell_i)_{i\geq 1})$ represents an analytic adelic line bundle $\widehat{\L}$. 
	
	The last part of the proposition follows from the fact that $\L|_A\cong L$.
\end{proof}

\subsection{Uniqueness of the construction}
Let $S$ be a smooth complex curve. As before, $A/S$ denotes an abelian scheme, and $\NN(A)^0/\Sbar$ denotes the identity component of the Néron model of $A/S$ over a smooth partial compactification $\Sbar\supset S$. Let $j\colon A\hookrightarrow \NN(A)^0$ be the inclusion. 
\begin{lem}
	The  restriction map
	\[
	j^*\colon \widehat{\Pic}(\NN(A)^0/\C)\to \widehat{\Pic}(A/\C).
	\]
	is injective.	
\end{lem}
\begin{proof}
	Since $A$ and $\NN(A)^0$ are regular  (hence normal) quasi-projective $\C$-schemes   and the inclusion $j$ induces an isomorphism of their function fields, we conclude the claim from \cite[Corollary 3.4.2]{YZ21}. We remark that the cited corollary uses the interpretation that adelic line bundles are analytic line bundles on suitable Berkovich analytic spaces.
\end{proof}
We observe a rigidified version of the lemma (the subsequent discussion can be clearly  placed in a more general setting). Let $e\colon \Sbar\to \NN(A)^0$ be the identity section of $\NN(A)^0/\Sbar$.
\begin{defn}\label{Def: rigidified adelic line bundle}
	A \textbf{rigidification} of an adelic line bundle $\widehat{\L}$ on $\NN(A)^0$ along $e$ is an isomorphism
	\[
	\phi\colon e^*\widehat{\L}\xrightarrow{\sim}\widehat{\O}_{\Sbar}
	\] of adelic line bundles on $\Sbar$, where $\widehat{\O}_{\Sbar}$ denotes the trivial adelic line bundle. The pair $(\widehat{\L},\phi)$ is called a \textbf{rigidified adelic line bundle}. One defines morphisms of rigidified adelic line bundles in a similar manner as in  Definition \ref{rigidified line bundle}. We let 
	\[
	\widehat{\Pic}(\NN(A)^0/\C)^e\subset \widehat{\Pic}(\NN(A)^0/\C)
	\] be the subgroup of isomorphism classes of rigidified adelic line bundles along $e$. Let $e_A=e|_S\colon S\to A$ be the identity section of $A/S$. We define in a similar way the subgroup
	\[
	\widehat{\Pic}(A/\C)^{e_A}\subset \widehat{\Pic}(A/\C)
	\] of isomorphism classes of  adelic line bundles on $A$ rigidified along $e_A$.
\end{defn}
\begin{rem}\label{rigidified underlying line bundle}
	If $\widehat{\L}$ is a rigidified adelic line bundle, then its underlying line bundle $\L$ admits an induced rigidification.
\end{rem}

\begin{lem}\label{injection of rigidified adelic line bundle}
	There exists a commutative diagram
	\[
	\begin{tikzcd}
		\widehat{\Pic}(\NN(A)^0/\C)  \ar[hook]{r}{j^*} & \widehat{\Pic}(A/\C) \\	
		\widehat{\Pic}(\NN(A)^0/\C)^e \ar[hook]{u}\ar{r}& \widehat{\Pic}(A/\C)^{e_A}\ar[hook]{u} 
	\end{tikzcd}
	\] Therefore, the lower horizontal map is also injective.
\end{lem}
\begin{proof}
	This follows immediately from the definition.
\end{proof}
\begin{prop}\label{uniqueness of adelic extension, geometric case}
	There is a unique rigidified adelic line bundle $\widehat{\L}$  on $\NN(A)^0$  such that $[n]^*\widehat{\L}=n^2\widehat{\L}$ for all $n\in \Z$.
\end{prop}
\begin{proof}
	By Proposition \ref{adelic extension, geometric case}, there is an analytic adelic line bundle $\widehat{\L}$ on $\NN(A)^0$ such that $[2]^*\widehat{\L}\cong 4\widehat{\L}$. Fix a rigidification $[0]^*\widehat{\L}\xrightarrow{\sim} \widehat{\O}_{\Sbar}$. Moreover,  the restriction of $\widehat{\L}$ over  $A$ yields the adelic line bundle~$\widehat{L}$ from \cite[Theorem 6.1.3]{YZ21} (denoted by $\overline{L}$ in loc.\ cit.), which is  uniquely determined  by the property  $[n]^*\widehat{L}=n^\epsilon\widehat{L}$ for all $n\in \Z$; in particular, taking $n=0$ shows that $\widehat{L}$ is rigidified along $e_A$. Combined with Lemma \ref{injection of rigidified adelic line bundle} and Remark \ref{continuous extension of metrized line bundle}, we see that there is at most one analytic adelic line bundle~$\widehat{\L}$ such that $[2]^*\widehat{\L}\cong 4\widehat{\L}$. To prove the last isomorphism  for all $n$, we proceed exactly as in \cite[Theorem 6.1.2(2)]{YZ21}. Let $n\in \Z\ohne\{0,1,2\}$. We have
	\[
	[2]^*[n]^*\widehat{\L}\cong 4[n]^*\widehat{\L}
	\] so that for $\widehat{\L}'\coloneq n^{-2}[n]^*\widehat{\L}$, 
	\[
	[2]^*\widehat{\L}'\cong 4\widehat{\L}'.
	\] Since $\widehat{\L}'$ is an adelic extension of $\L$, by the uniqueness for $n=2$, we must have $\widehat{\L}'\cong\widehat{\L}$, i.e.\ $[n]^*\widehat{\L}\cong n^2\widehat{\L}$. 
\end{proof}

\begin{rem}[admissible metrized extension over the Néron model]\label{continuous extension of metrized line bundle}
	Let $(\L,(\X_i,\overline{\LL_i},\ell_i)_{i\geq 1})$ be a tuple representing the analytic adelic line bundle~$\widehat{\L}$ on $\U\coloneq \NN(A)^0$ from Proposition \ref{adelic extension, geometric case}. Recall $\L|_A\xrightarrow{\sim} L$ as rigidified line bundles (and~$L$ was symmetric and relatively ample). Let $\|\cdot\|_A$ be the smooth admissible metric on~$L$, normalized using the rigidification of $L$, and write  $\overline{L}=(L,\|\cdot\|_A)$. By construction,  $\overline{\LL_i}|_A\cong \overline{L}$ is an isometry for every $i\geq 1$. By the analytic density $A^\an\subset \U^\an$, $\widehat{\L}$ induces a unique continuous metric $\|\cdot\|_\U$ on $\L$ such that
	\[
	[2]^*(\L,\|\cdot\|_\U)\cong 4 (\L,\|\cdot\|_\U)
	\] and such that $e^*\|\cdot\|_U$ coincides with a prescribed metric on $\O_{\Sbar}$. In other words, $(\L,\|\cdot\|_\U)$ is an admissible metrized line bundle extending $(L,\|\cdot\|_A)$.
\end{rem}

\subsection{Positivity (or integrability) of the adelic extension}
Recall  the positivity notions in \ref{nefness and integrability}.
\begin{defn}
	Let $\widehat{\L}$ be an adelic line bundle on an integral scheme $\U$ that is quasi-projective over $\C$. We call $\widehat{\L}$ \textbf{strongly nef} if it is isomorphic to an adelic line bundle on~$\U$ represented by  a tuple $(\L,(\X_i,\overline{\LL_i},\ell_i)_{i\geq 1})$, where the metrized line bundles $\overline{\LL_i}$ on $\X_i$ are nef for all $i\geq 1$. Informally,   $\widehat{\L}$  is said to be strongly nef if it is the limit of nef model line bundles. The adelic line bundle $\widehat{\L}$ is called \textbf{nef} if there exists a strongly nef adelic line bundle $\widehat{N}$ on $\U$ and a sequence $(r_i)_{i\geq 1}$ of positive rational numbers converging to zero such that $\widehat{\L}+r_i\widehat{N}$ is strongly nef for all $i\geq 1$. An adelic line bundle is called \textbf{integrable} if it is a difference of two  nef adelic line bundles. 
\end{defn}
\begin{rem}\label{integrability using just nefness}
	Suppose $\widehat{\L}$ is an integrable adelic line bundle on $\U$. Then it can be written as the difference of two \emph{strongly} nef adelic line bundles. Thus, our definition of integrability is equivalent to that of Yuan--Zhang \cite[\S 2.5.2]{YZ21}. Indeed, suppose $\widehat{\L}=\widehat{\L_1}-\widehat{\L_2}$ for some nef adelic line bundles $\widehat{\L_i}$. By definition, there are strongly nef adelic line bundles $\widehat{N_i}$ such that $\widehat{\L_i}+\widehat{N_i}$ are strongly nef for $i=1,2$. Then
	\[
	\widehat{\L}\cong(\widehat{\L_1}+\widehat{N_1}+\widehat{N_2})-(\widehat{\L_2}+\widehat{N_2}+\widehat{N_1})
	\] is a decomposition of $\widehat{L}$ into two strongly nef adelic line bundles.
\end{rem}
Let $A/S,\, \NN(A)^0/\Sbar$ be as before,  but now we let $L$  be  either a symmetric or an antisymmetric line bundle on $A$, rigidified along the identity section of $A/S$. By Lemma \ref{equivalence of cat. of rigidified line bundle}, we still have a rigidified symmetric respectively antisymmetric line bundle $\L$ on $\NN(A)^0$.
\begin{lem}\label{nef adelic extension, geometric case}
	Suppose $\L$ is symmetric and relatively ample. The analytic adelic line bundle $\widehat{\L}$ from Proposition \ref{uniqueness of adelic extension, geometric case} is nef.
\end{lem}
\begin{proof}
	The proof idea is the same as in \cite[Theorem 6.1.1]{YZ21}. Since $\L$ is relatively ample, we can use  Lemma \ref{extension of line bundle, positivity, anayltic} to find a projective  model $\X$ of $\NN(A)^0$, a metrized $\Q$-line bundle $\overline{\LL}$ on $\X$ and a metrized line bundle $\overline{N}$ coming from the base such that  $\overline{\LL'}\coloneq \overline{\LL}+\overline{N}$ is nef. We  run the construction in \ref{Sequence of projective models, geometric case} with the initial input $(\X,\overline{\LL})$ and obtain a tuple $(\L,(\X_i,\overline{\LL_i},\ell_i)_{i\geq 1})$ representing an adelic line bundle $\widehat{\L}$. Note that the $f_i$'s in loc.\ cit.\ are projective morphisms, hence they preserve nefness. With the beginning  nef input $\overline{\LL'}$, each 
	\[
	\frac{1}{4^i}f_i^*\overline{\LL'}=\overline{\LL_i}+\frac{1}{4^i}\overline{N}\in \Pic(\X_i)_\Q
	\] is nef. This implies that for any positive integer $a$, the line bundle $\overline{\LL_{i+a}}+2^{-a}\overline{N}$ is nef. Note that for any positive integer $a$, the tuple $(\L,(\X_{i+a},\overline{\LL_{i+a}},\ell_{i+a})_{i\geq 1})$  represents the same adelic line bundle~$\widehat{\L}$, and by construction, $\widehat{\L}+2^{-a}\overline{N}$ is a limit of nef model line bundles, i.e.\ it is strongly nef. It follows that $\widehat{\L}$ is nef. 
\end{proof}
\begin{lem}\label{integrabl adelic extension, symmetric, geometric}
	Suppose $\L$ is symmetric. There is an  integrable analytic adelic extension~$\widehat{\L}$ of $\L$, which is unique with the property $[n]^*\widehat{\L}\cong n^2\widehat{\L}$ for all $n\in \Z$.
\end{lem}
\begin{proof}
	The uniqueness follows from Proposition \ref{uniqueness of adelic extension, geometric case}, once we have constructed one adelic extension. Observe that $\L\cong \L^+-\L^-$ can be decomposed into a difference of two relatively ample  and \emph{symmetric} line bundles $\L^{\pm}$ on $\NN(A)^0$. In fact, take a relatively ample line bundle~$H_0$ (which exists because $\NN(A)^0/\Sbar$ is quasi-projective) and consider  $H\coloneq H_0+[-1]^*H_0$,  which is  relatively ample and symmetric. There is a positive integer~$n$ such that $\L+nH$ is relatively ample (Lemma \ref{GW, 13.62}). Set  $\L^+=\L+nH,\, \L^-=nH$. By Lemma \ref{nef adelic extension, geometric case}, $\L^{+}$ and $\L^-$ admit nef adelic extensions. Therefore, the difference of these adelic extensions is an integrable adelic extension of $\L$.
\end{proof}
Since the (extended) Poincaré bundle is symmetric (Remark \ref{symmetry of Poincare bundle}), we deduce
\begin{prop}[adelic prolongation of extended Poincaré bundle]\label{adelic prolongation of extended Poincaré bundle, geometric case}
	The extended rigidified Poincaré bundle 	 $\overline{\PP}$  on $\NN(A)\times_{{\Sbar}} \NN(A^\vee)$ admits  an integrable adelic extension~$\widehat{\overline{\PP}}$, which is unique  with the dynamical property that
	\[
	[n]^*\widehat{\overline{\PP}}\cong n^2\widehat{\overline{\PP}}
	\] for all $n\in \Z$.
\end{prop}
\begin{lem}\label{integrabl adelic extension, antisymmetric, geometric}
	Suppose $\L$ is antisymmetric. There is an integrable analytic adelic extension $\widehat{\L}$, which is unique with the property that $[n]^*\widehat{\L}\cong n\widehat{\L}$ for all $n\in \Z$.
\end{lem}
\begin{proof}
	Let $\PP$ be the Poincaré bundle on the abelian scheme $A\times_S A^\vee$.  Recall  $\L|_A\cong L$ as rigidified line bundles, and we can view $L\in A^\vee(S)$; denote by $\sigma\colon S\to A^\vee$ the corresponding section so that
	\begin{equation}\label{section of rigidified line bundle}
		L\cong (1,\sigma\circ \pi)^*\PP
	\end{equation}
	as rigidified line bundles, where $\pi\colon A\to S$ is the structure morphism. Let $\overline{\PP}$ be the extended Poincaré bundle on $\NN(A)^0\times_{{\Sbar}} \NN(A^\vee)^0$, rigidified along the identity. By Lemma \ref{equivalence of cat. of rigidified line bundle}, for some (large) integer $n>0$, $n\sigma$ extends to a section $\bar{\sigma}\colon \Sbar\to \NN(A)^0$ corresponding to the rigidified line bundle $\L$, and the isomorphism \eqref{section of rigidified line bundle} extends to an isomorphism
	\begin{equation}
		\L\cong (1,\bar{\sigma}\circ \bar{\pi})^*\overline{\PP},
	\end{equation} where $\bar{\pi}\colon \NN(A)^0\to \Sbar$ denotes the structure morphism.
	
	For $m\in \Z$ let $[m]$ be the total multiplication on $\NN(A)\times_{{\Sbar}} \NN(A)^0$ (which is an $\Sbar$-morphism), and $[m]'=([m]_{\NN(A)^0},1)$ be the partial multiplication on the first component (which is an $\NN(A^\vee)^0$-morphism). Let $\widehat{\overline{\PP}}$ be the integrable adelic extension  on $\NN(A)^0\times_{{\Sbar}} \NN(A^\vee)^0$ obtained in Proposition \ref{adelic prolongation of extended Poincaré bundle, geometric case}. It follows that
	\[
	[2]'^*\widehat{\overline{\PP}}\cong 2\widehat{\overline{\PP}}.
	\] Consider 
	\[
	\widehat{\L}'\coloneq (1,\bar{\sigma}\circ \bar{\pi})^*\widehat{\overline{\PP}}, 
	\] which is an analytic adelic extension of $\L$ satisfying
	\[
	[2]^*\widehat{\L}'\cong 2\widehat{\L}'.
	\]  By the uniqueness statement, Proposition \ref{uniqueness of adelic extension, geometric case}, we must have
	\[
	\widehat{\L}'\cong \widehat{\L}.
	\] Since $\widehat{\overline{\PP}}$ is integrable, and the integrability is preserved under pullback (see \cite[\S 2.5.5]{YZ21}), we conclude that $\widehat{\L}$ is integrable. The uniqueness of $\widehat{\L}$ follows from the uniqueness of~$\widehat{\overline{\PP}}$.
\end{proof}
For convenience, we summarize our results in this section as follows.
\begin{thm}\label{adelic extension summary, analytic case}
	Let $A/S$ be an abelian scheme over a smooth $\C$-curve $S$. Let $S\subset \Sbar$ be a smooth partial compactification and let $\NN(A)^0/\Sbar$ be the identity component of the Néron model of $A/S$ over $\Sbar$. Let $L$ be  either a symmetric or an antisymmetric rigidified line bundle on $A$ with the unique rigidified extension $\L$ on $\NN(A)^0$. Suppose $L$ is endowed with an admissible metric, normalized along the identity section. Then  $\L$ admits an integrable adelic extension $\widehat{\L}$, which is unique with the dynamical property that
	\[
	[n]^*\widehat{\L}\cong n^\epsilon\widehat{\L}\quad\text{for all }n\in \Z,
	\] and which extends the isomorphism $[n]^*\widehat{L}\cong n^\epsilon \widehat{L}$ from \cite[Theorem 6.1.3]{YZ21} for all $n\in \Z$; here $\epsilon=2$ if $L$ is symmetric, and $\epsilon=1$ if $L$ is antisymmetric. Moreover, if $L$ is relatively ample, then  $\widehat{\L}$ is nef.
\end{thm}
\begin{rem}
	By a similar, and in fact simpler, argument (``forgetting the metrics''), we
	obtain analogues of Theorem  \ref{adelic extension summary, analytic case} over an arbitrary base field (with the trivial  valuation).
\end{rem}
\begin{rem}[higher dimensional base]
	The  construction  of the adelic line bundle $\widehat{\L}$ in Theorem \ref{adelic extension summary, analytic case} only requires the following properties of $\NN(A)^0$ (but not its universal property):
	\begin{enumerate}
		\item $\NN(A)^0/\Sbar$ is quasi-projective;
		\item $L$ extends uniquely to $\L$ over $\NN(A)^0$ as a rigidified line bundle.
	\end{enumerate}
	This suggests that a similar construction should also work over a higher-dimensional base. Let $S$ be a smooth quasi-projective $k$-variety, where $k$ is either a field or $k=\C$ equipped with the standard archimedean valuation. Suppose there exists a smooth commutative group scheme $\mathcal{G}/\Sbar$ extending an abelian scheme $A/S$, where $S\subset \Sbar$   is a partial  compactification and $\Sbar$ is a normal $k$-variety. Assume moreover that $\mathcal{G}/\Sbar$ satisfies the properties listed in the above display. For instance, such $\mathcal{G}/\Sbar$ exists when $A/S$ has semistable reduction  over $\Sbar$ with connected fibers and $L$ is symmetric line bundle on $A$, rigidified along the identity section \emph{cf}.\ \cite[\S II.3]{Ast129}. In general, a projective model  of  $\mathcal{G}$ over $k$, constructed using a similar idea as in Lemma \ref{extension of line bundle, positivity}, need not be flat over  $\Sbar$. On the other hand, flatness does not seem to be essential for our constructions leading to Theorem \ref{adelic extension summary, analytic case}.
\end{rem}

\section{Arithmetic case}\label{Arithmetic Case}
Now let $S$ be a smooth curve over a number field~$K$. Denote by $O_K$  the ring of algebraic integers of~$K$.
To prove the analogous adelic extension result to Theorem \ref{adelic extension summary, analytic case}, we  need to find suitable (quasi-)projective models for $\NN(A)^0$. We explain the setup through the following technical lemma, which follows from spreading out arguments, see e.g.\ \cite[Theorem 3.2.1]{Qpoint}. See  Table 1 in \cite[Appendix C.1]{Qpoint} for  properties of morphisms that can be spread out.
\begin{lem}\label{Lemma on finding models}
	Let  $p\colon X\to S$ be a flat finite type morphism from an integral scheme $X$. Let~$L$ be a $p$-ample line bundle on $X$. Let $f\colon X\to X$ be a quasi-finite $S$-morphism. There exist integral quasi-projective $O_K$-schemes $\U$ and $\V$,  integral projective $O_K$-schemes $\X$ and $\SS$, and  a commutative diagram 
	\[
	\begin{tikzcd}
		X\ar[hook]{r} \ar{d}{p}& \mathcal{U} \ar[hook]{r} \ar{d} & \X \ar{d}{\pi}\\
		S \ar[hook]{r} & \mathcal{V} \ar[hook]{r} & \SS
	\end{tikzcd}
	\] with the following properties. 
	\begin{enumerate}
		\item $\V$ is a quasi-projective model of $S$ (Definition \ref{quasi-projective model, arithmetic}) and  $X\hookrightarrow \U$ is a monomorphism.
		\item $\U\hookrightarrow \X$ and $\V\hookrightarrow \SS$ are dense open immersions.
		\item The morphism $\pi$ is projective and flat.
		\item The restriction $\pi|_\U\colon \U\to \V$ is a flat quasi-projective  morphism. 
		\item  The line bundle $L$ extends to a $\pi$-ample $\Q$-line bundle $\LL$ on $\X$.
		\item 	 The $S$-morphism $f$ extends to a quasi-finite $\mathcal{V}$-morphism $f_\V\colon \U\to \U$.
	\end{enumerate}
\end{lem}
\begin{proof}
	Recall that $S$ is quasi-projective over~$K$ by \cite[Theorem 15.18]{GW}. After spreading out the morphism $S\to \Spec K$, we obtain a dense open subscheme $V\subset \Spec O_K$ and a quasi-projective   $V$-scheme $\V$ along with an open immersion $S\hookrightarrow \V_K\coloneq \V\times_V \Spec K$. Indeed, at this point we can take $V = \Spec O_K$, but  we may want to shrink $V$ and $\V$ in the later in the argument. By the relative ampleness of $L$, there is a positive multiple $mL$ that  induces a locally closed $S$-immersion $X\hookrightarrow \P^n_S$ for some $n\geq 0$ (Lemma \ref{GW, 13.62}). The induced monomorphism $S\hookrightarrow \V$ gives rise to a monomorphism 
	\[
	X\hookrightarrow \P^n_S\hookrightarrow \P^n_\V.
	\] Let $\U$ be the scheme-theoretic closure of $X$ in $\P^n_\V$. Then we obtain a monomorphism $X\hookrightarrow \U$, and the induced morphism $\U\to \V$ is quasi-projective (even projective in this case). Moreover, restriction of the tautological line bundle $\O(1)$ on $\P^n_\V$ along the closed $\V$-immersion $\U\hookrightarrow \P^n_\V$ yields an line bundle $\LL_\U$ which is relatively ample and extends $mL$.

	By  Raynaud--Gruson flattening (see e.g.\ \cite[Theorem 14.143]{GW}), there exist blow-ups $\U'\to \U$ resp.\ $\V'\to \V$ which are isomorphisms over $X$ resp.\ over $S$ such that $\U'\to\V'$ is flat. Alternatively, since $S$ is Dedekind and flatness can be spread out, by shrinking $\V$, we also obtain the flatness. By \cite[Proposition 13.64]{GW}, the pullback of $\LL_U$ along $\U'\to \U$ is relatively ample w.r.t.\ $\U'/\V'$.  Relabelling, we obtain a projective flat morphism $\U\to \V$ where $\U, \V$ satisfy the properties in Item 1 of the lemma with the additional property that there exists a $\V$-ample line bundle $\LL_U$ on $\U$ extending $mL$.
	
	Similarly, we use a positive multiple $m'\LL_U$ to realize a closed $\V$-immersion $\U\hookrightarrow \P^{n'}_\V$ for some $n'\geq 0$. Let $\SS$ be a projective closure of $\V$ over $O_K$. Take $\X$ to be the scheme-theoretic closure along the composition of the immersions
	\[
	\U\hookrightarrow \P^{n'}_\V\hookrightarrow \P^{n'}_\SS.
	\] Then the induced morphism $\pi\colon \X\to \SS$ is projective. We can assume that $\pi$ is flat by blowing ups $\X'\to \X$ resp.\ $\SS'\to \SS$ which are isomorphisms over $\U$ resp.\ over $\V$ by  Raynaud--Gruson flattening, and then relabel them. As before, this procedure yields  a $\pi$-ample line bundle $\LL$ that extends $m'\LL_\U$.
	
	Until now we have settled Items 1--5 of the lemma. The last item follows by another spreading out argument. More precisely, we know that quasi-finiteness can be spread out,   thus~$f$ extends to a $\V_0$-morphism $\U_0\to \U_0$ for a  nonempty (possibly smaller) open subschemes $\V_0\subset \V$ and $\U_0\subset \U$, where $\V_0$ is defined over some open subset of $\Spec O_K$. The resulting morphism $\U_0\to \V_0$ is  flat and quasi-projective. Relabelling $\U_0$ as $\U$ and $\V_0$ as $\V$,  we obtain the final item.
\end{proof}
Let $A/S$ be an abelian scheme. Let $[n]\colon A\to A$ be the multiplication-by-$n$ map for some $n\in \N\ohne\{0,1\}$, which is a finite $S$-morphism. Let $S\subset\Sbar$ be a smooth partial compactification over $K$, and let $\NN(A)^0/\Sbar$ be the identity component of the Néron model of $A/S$.
\begin{lem}
	The $S$-morphism $[n]$ extends to a quasi-finite $\Sbar$ morphism $[n]\colon \NN(A)^0\to \NN(A)^0$.
\end{lem}
\begin{proof}
	By the Néron mapping property, the $S$-morphism $[n]$ on $A$ extends uniquely to an $\Sbar$-morphism on $\NN(A)^0$. It is quasi-finite, in fact étale, by \cite[\S 7.3, Lemma 2(2)]{BLR}, since $\mathrm{char}\, \kappa(s)\nmid n$ for all $s\in \Sbar$.
\end{proof}

\begin{lem}\label{quasi-projective model for arithmetic Neron model}
	Let $\L$ be a relatively ample line bundle on $\NN(A)^0/\Sbar$. Let $f=[2]$ be the quasi-finite $\Sbar$-morphism on $\NN(A)^0$. There exist integral quasi-projective $O_K$-schemes $\U$ and~$\V$, integral projective $O_K$-schemes $\X$ and $\SS$, and  a commutative diagram 
	\[
	\begin{tikzcd}
		\NN(A)^0\ar[hook]{r} \ar{d}{p}& \mathcal{U} \ar[hook]{r} \ar{d} & \X \ar{d}{\pi}\\
		\Sbar \ar[hook]{r} & \mathcal{V} \ar[hook]{r} & \SS
	\end{tikzcd}
	\] with the following properties. 
	\begin{enumerate}
		\item $\V$ is a quasi-projective model of $\Sbar$ and $\NN(A)^0\hookrightarrow \U$ is a monomorphism.
		\item $\U\hookrightarrow \X$ and $\V\hookrightarrow \SS$ are dense open immersions. 
		\item The morphism $\pi$ is projective and flat.
		\item The restriction $\pi|_\U\colon \U\to \V$ is a flat quasi-projective  morphism. 
		\item The line bundle $\L$ extends to a  $\pi$-ample $\Q$-line bundle $\LL$ on $\X$.
		\item  The $\Sbar$-morphism $f\colon \NN(A)^0\to \NN(A)^0$ extends to a quasi-finite $\mathcal{V}$-morphism $f_\V\colon \U\to \U$ of smooth quasi-projective $\V$-group scheme $\U$.
	\end{enumerate}
\end{lem} 
\begin{proof}
	Items 1--6 are applications of the two lemmas preceding the current lemma, except for the claim about the group scheme structure. However, this  follows directly by another spreading out argument after  shrinking  the sources of the targets of the morphisms in question, since quasi-projectivity, smoothness, and the group law all spread out.
\end{proof}
\begin{rem}\label{def of quasi-projective models}
	Note that $\U$ resp.\ $\X$  in Lemma \ref{quasi-projective model for arithmetic Neron model} are quasi-projective resp.\ projective $\Z$-schemes via the finite (projective) morphism $\Spec O_K\to \Spec \Z$, and they can be regarded as {quasi-projective} resp.\ {projective models} of $\NN(A)^0$ over $\Z$ in the sense we defined before (\emph{cf}.\ Definitions \ref{projective model, adelic setting} and \ref{quasi-projective model, arithmetic} with $K=\Q$).
\end{rem}
Recall that the purpose for Lemma \ref{extension of line bundle, positivity} was to use a suitable positivity property of a model line bundle to provide an initial input for constructing the sequence of models. We  give an arithmetic analogue. 
\begin{defn}\label{projective model, arithmetic}
An integral scheme $\X$ that is flat and separated  over $O_K$ with smooth generic fiber is called an \textbf{arithmetic variety}; it is called a (quasi-)projective arithmetic variety, if the structure morphism $\X\to \Spec O_K$ is (quasi-)projective.
\end{defn}

Recall the corresponding positivity notion, \emph{cf}.\ \cite[Definition 5.38]{Moriwaki}.
\begin{defn}\label{def. of arith. nef}
	Let $\X$ be a projective arithmetic variety  and $\LL$ be a line bundle on $\X$ endowed with a continuous metric $\|\cdot\|$ on $\LL^\an$ over the complex manifold $\X^\an=\X(\C)$. We call $(\LL,\|\cdot\|)$ \textbf{arithmetically nef} if 
	\begin{itemize}
		\item $\LL$ is relatively nef w.r.t.\ $\X\to \Spec\Z$;
		\item $\|\cdot\|$ is semipositive;
		\item the associated height function $h_{(\LL,\|\cdot\|)}$  (see e.g.\ \cite[\S 9.1]{Moriwaki}) on $\X(\overline{\Q})$ is non-negative.  
	\end{itemize}
\end{defn}

\begin{lem}\label{extension of line bundle, positivity, arithmetic case}
	In Lemma \ref{quasi-projective model for arithmetic Neron model}, we can replace
	$\X$ by an integral projective  $O_K$-scheme~$\X'$ and replace the
	$\Q$-line bundle $\LL$ on $\X$ by a $\Q$-line bundle $\LL'$ on $\X'$ in such a way that
	\begin{enumerate}
		\item the $\Q$-line bundle $\LL'$ extends $\L$;
		\item $\LL'$ admits a continuous metric with the following property: if
		$\overline{\LL'}$ denotes the resulting metrized line bundle, then there exists
		an arithmetically nef line bundle $\overline{\mathscr{M}}$ on $\SS$ such that
		\[
		\overline{\LL'}+(\pi')^*\overline{\mathscr{M}}
		\]
		is arithmetically nef on $\X'$, where $\pi'\colon \X'\to \SS$ is the new
		structure morphism.
	\end{enumerate}
\end{lem}
\begin{proof}
	The proof is similar to that of Lemma \ref{extension of line bundle, positivity}. Let $\LL_\U\coloneq \LL|_\U$, which is a $\V$-ample line bundle. The relative ampleness of $\L_\U$ and Lemma \ref{GW, 13.65} imply that there is an ample line bundle $\mathscr{M}$ on $\SS$ such that	 some positive multiple of $\L_\U+(\pi|_\U)^*\mathcal{M}$, where $\mathcal{M}\coloneq \mathscr{M}|_\V$, induces a locally closed immersion of $\U\hookrightarrow \P^n_{O_K}$ for some $n\geq 0$; let $\X'$ be the scheme-theoretic closure of $\U$ in $\P^n_{O_K}$. 
	
	By construction, we have a rational map $\X'\dashrightarrow \SS$. Upon blowing up $\X'$ further, we can turn it into an actual morphism $\pi'\colon \X'\to \SS$, which is still projective. By  Raynaud--Gruson flattening, we may assume that $\pi'$ is  flat. Note $\pi'|_\U=\pi|_\U$.
	
	Let $\O(1)$ be the line bundle on $\X$ obtained by restricting  the tautological line bundle on~$\P^n_{O_K}$ along the closed immersion $\X'\hookrightarrow \P^n_{O_K}$. By design, $\O(1)|_\U$ is isomorphic to a   $m(\LL_\U+(\pi|_\U)^*\mathcal{M})$ for some positive integer $m$, thus the (absolutely) ample $\Q$-line bundle $\LL_1\coloneq\frac{1}{m}\O(1)$  on $\X'$  extends $\LL_\U+(\pi|_\U)^*\mathcal{M}$. By ampleness, we can endow $\LL_1$ and $\mathscr{M}$ with semipositive metrics such that both becomes arithmetically nef; denote the resulting metrized line bundles by $\overline{\LL_1}$ resp.\ $\overline{\mathscr{M}}$. The desired new metrized $\Q$-line bundle $\overline{\LL'}$  is given by $\overline{\LL_1}-\pi^*\overline{\mathscr{M}}$. By construction, $\LL'|_\U\xrightarrow{\sim} \L$ as $\Q$-line bundles.
\end{proof}

The construction for a tuple representing the desired adelic extension is similar to the complex-analytic case.
\begin{conv}
	As in the complex-analytic case, in the following we assume all projective models of $\NN(A)^0$  to be normal.
\end{conv}
Relabel the pair $(\X',\overline{\LL'})$ from  Lemma \ref{extension of line bundle, positivity, arithmetic case} as $(\X,\overline{\LL})$, and continue with the notation of that lemma. Write $\LL_U=\LL|_\U$. We have a quasi-finite $\V$-morphism $f_\V=[2]\colon \U\to \U$ (Lemma \ref{quasi-projective model for arithmetic Neron model}). 
\begin{void}[sequence of models]
	With the above initial input and by resolving the indeterminacy loci (see \ref{Resolve indeterminancy locus}), we obtain a sequence of projective models $(\X_i)_{i\geq 1}$ of $\U$ by a similar construction to \ref{Sequence of projective models, geometric case}. Indeed, let $\X_0=\X$. The commutativity of the right square in Lemma \ref{quasi-projective model for arithmetic Neron model} implies that there exists  an induced open immersion $\U\times_\V\U\hookrightarrow \X_0\times_{\SS}\X_0$; let $\X_1\subset \X_0\times_{\SS} \X_0$ be the scheme-theoretic closure of the graph $\Gamma_{f_\V}\subset \U\times_\V\U$. Then the iteration procedure as in the complex-analytic case carries over and produces a sequence of projective models $(\X_i)_{i\geq 1}$ of $\U$ over $\Z$.
	
	In a similar vein, we obtain a sequence of   $\Q$-line bundles $(\LL_i)_{i\geq 1}$ on $\X_i$ such that
	\begin{itemize}
		\item	$\ell_i\colon \LL_\U\xrightarrow{\sim} \LL_i|_\U$;
		\item 	$f_{\V}^*\LL_\U\cong 4\LL_\U$.
	\end{itemize}
	As in \ref{metrized model line bundle}, each $\LL_i$ is endowed with a smooth metric $\|\cdot\|_i$ such that
	\[
	(\LL_i,\|\cdot\|_i)\cong \frac{1}{4}f_i^*\left(\LL_{i-1},\|\cdot\|_{i-1}\right)\in \overline{\Pic}(\X_i)_\Q
	\] is an isometry for every $i\geq 1$, where $f_i\colon \X_i\to \X_{i-1}$ is the resolved morphism such that $f_i|_\U=f_\V=[2]$.
\end{void}
With a similar proof as in Proposition \ref{adelic extension, geometric case}, we obtain the following extension result in the arithmetic situation.
\begin{prop}\label{adelic extension, arithmetic case}
	The sequence $(\LL_\U,(\X_i,\overline{\LL_i},\ell_i)_{i\geq 1})$ represents an arithmetic adelic line bundle on~$\U$. Then, by Definition \ref{arithmetic adelic divisors and line bundles on quasi-projective $K$-varieties}, we formally obtain an adelic extension $\widehat{\L}\in \widehat{\Pic}(\NN(A)^0/\Z)$ of~$\L$. 
\end{prop}

Let $j\colon A\hookrightarrow \NN(A)^0$ be the inclusion, $e\colon  \Sbar\to \NN(A)^0$ and $e_A=e|_A\colon S\to A$ the identity sections of the respective group schemes. One defines the group of isomorphism classes of rigidified adelic line bundles in the arithmetic setting completely analogously to the complex-analytic case (Definition \ref{Def: rigidified adelic line bundle}).  
\begin{rem}\label{Rmk: rigidified adelic line bundle}
	Let $\widehat{\L}$ be an adelic line bundle on $\NN(A)^0$, rigidified along the identity section of $\NN(A)^0/\Sbar$. By Definition \ref{arithmetic adelic divisors and line bundles on quasi-projective $K$-varieties} and by a spreading out argument, one can find  quasi-projective $\Z$-models $\U\supset \NN(A)^0$ and  $\V\supset \Sbar$  as in Lemma \ref{quasi-projective model for arithmetic Neron model}, and an adelic line bundle $\widehat{\L}_\U$ on $\U$ such that $\widehat{\L}_\U|_{\NN(A)^0}\cong \widehat{\L}$ and $\widehat{\L}_\U$ is rigidified along the identity section of $\U/\V$.  
\end{rem}
The following lemma is analogous to Lemma \ref{injection of rigidified adelic line bundle} and is proved similarly. 
\begin{lem}
	The  restriction map
	\[
	j^*\colon \widehat{\Pic}(\NN(A)^0/\Z)\to \widehat{\Pic}(A/\Z).
	\]
	is injective.	
\end{lem}	
\begin{proof}
	The claim follows directly from \cite[Corollary 3.4.2]{YZ21}: the schemes $\NN(A)^0$ and~$A$ are normal, essentially projective and flat over $\Z$, and the inclusion~$j$ induces an isomorphism  of the function fields $K(A)\cong K(\NN(A)^0)$ (recall that $K$ denotes a number field, which is finitely generated over $\Z$, following   the convention in \cite[\S 1.5]{YZ21}.)
\end{proof}	
\begin{lem}\label{injection of rigidified adelic line bundle, arithmetic}
	There exists a commutative diagram
	\[
	\begin{tikzcd}
		\widehat{\Pic}(\NN(A)^0/\Z)  \ar[hook]{r}{j^*} & \widehat{\Pic}(A/\Z) \\	
		\widehat{\Pic}(\NN(A)^0/\Z)^e \ar[hook]{u}\ar{r}& \widehat{\Pic}(A/\Z)^{e_A}\ar[hook]{u} 
	\end{tikzcd}
	\] Therefore, the lower horizontal map is also injective.
\end{lem}
\begin{proof}
	This follows immediately from the definition.
\end{proof}

\begin{prop}
	The adelic extension $\widehat{\L}\in \widehat{\Pic}(\NN(A)^0/\Z)$ obtained in Proposition \ref{adelic extension, arithmetic case} is unique with the property that $[n]^*\widehat{\L}=n^2\widehat{\L}$ for all $n\in \Z$.
\end{prop}
\begin{proof}
	By Lemma \ref{injection of rigidified adelic line bundle, arithmetic} and an argument similar to that in Proposition \ref{uniqueness of adelic extension, geometric case} (choosing a quasi-projective model $\NN(A)^0\subset \U$ as in Lemma \ref{quasi-projective model for arithmetic Neron model}), we reduce the uniqueness statement to the corresponding uniqueness statement in the arithmetic case proved by Yuan--Zhang \cite[Theorem 6.1.3]{YZ21}.
\end{proof}
Finally, we want to show that the adelic extension $\widehat{\L}$ is integrable.
\begin{defn}
	Let $\widehat{\L}_\U$ be an adelic line bundle on a  quasi-projective arithmetic variety~$\U$ (Definition \ref{projective model, arithmetic}). We call $\widehat{\L}_\U$ \textbf{strongly nef} if it is isomorphic to an adelic line bundle on $\U$ represented by a tuple $(\L_\U,(\X_i,\overline{\LL_i},\ell_i)_{i\geq 1})$, where  the metrized $\Q$-line bundles $\overline{\LL_i}$ on $\X_i$ are arithmetically nef for all $i\geq 1$. Informally,   $\widehat{\L}_\U$  is said to be strongly nef if it is the limit of nef model line bundles. It is called \textbf{nef} if there exists a strongly nef adelic line bundle $\widehat{N}$ on $\U$ and a sequence $(r_i)_{i\geq 1}$ of positive rational numbers converging to zero such that $\widehat{\L}_\U+r_i\widehat{N}$ is strongly nef for all $i\geq 1$. An adelic line bundle is called \textbf{integrable} if it is a difference of two  nef adelic line bundles. 
\end{defn}
Let $X$ be a smooth quasi-projective variety over a number field $K$, and let $\widehat{\L}$ be an adelic line bundle on $X$. By Definition \ref{arithmetic adelic divisors and line bundles on quasi-projective $K$-varieties}, there is a quasi-projective model $\U$ of $X$ and an adelic line bundle $\widehat{\L}_\U$ on $\U$ such that $\widehat{\L}_\U|_X\cong \widehat{\L}$. 
\begin{defn}
	We call an adelic line bundle $\widehat{\L}$ on $X$ \textbf{strongly nef} resp.\ \textbf{nef} resp.\ \textbf{integrable} if there is a quasi-projective model $\U\supset X$ and a strong nef resp.\ nef.\ resp.\ integrable adelic line bundle $\widehat{\L}_\U$ on $\U$ such that $\widehat{\L}_\U|_X\cong \widehat{\L}$.
\end{defn}
Resume the notation that $A/S$ is an abelian scheme over a smooth curve $S$ over a number field $K$; let $S\subset \Sbar$ be a smooth partial compactification. Let $L$ be a rigidified symmetric $S$-ample line bundle on $A$. Let $\L$ be the unique  line bundle on $\NN(A)^0$ extending $L$ with similar properties (Lemma \ref{ample extension}).
\begin{lem}
	The adelic extension $\widehat{\L}$ from Proposition \ref{adelic extension, arithmetic case} is nef.
\end{lem}
\begin{proof}
	Let $\U$ be a quasi-projective model of $\NN(A)^0$ as constructed in Lemma \ref{quasi-projective model for arithmetic Neron model}.  A direct modification of the proof of Proposition \ref{nef adelic extension, geometric case}, together with  Lemma \ref{extension of line bundle, positivity, arithmetic case}, shows that there is a nef adelic line bundle~$\widehat{\L}_\U$ on $\U$ such that $\widehat{\L}_\U|_{\NN(A)^0}\cong \widehat{\L}$. By  definition, this means that $\widehat{\L}$ itself is nef.
\end{proof}
Then we deduce the general case as in the complex-analytic situation (see the claims leading to Theorem \ref{adelic extension summary, analytic case}). Let $L$ be either a symmetric or an antisymmetric rigidified line bundle on $A/S$. Let $\L$ be the unique extension of $L$ with similar properties.
\begin{prop}
	The adelic extension $\widehat{\L}$ from Proposition \ref{adelic extension, arithmetic case} is integrable.
\end{prop}
Let us summarize our findings in this section.
\begin{thm}\label{adelic extension summary, arithmetic case}
	Let $A/S$ be an abelian scheme over a smooth curve $S$ over a number field $K$. Let $S\subset \Sbar$ be a smooth partial compactification over $K$ and let $\NN(A)^0/\Sbar$ be the identity component of the Néron model of $A/S$ over $\Sbar$. Let $L$ be  either a symmetric or an antisymmetric rigidified line bundle on $A$ with the unique rigidified extension $\L$ on~$\NN(A)^0$. Suppose $L$ is endowed with an admissible metric, normalized along the identity section. There exists an integrable arithmetic adelic extension $\widehat{\L}$ of $\L$, which is unique with the dynamical property that
	\[
	[n]^*\widehat{\L}\cong n^\epsilon\widehat{\L}\quad\text{for all }n\in \Z,
	\] and which extends the isomorphism $[n]^*\widehat{L}\cong n^\epsilon \widehat{L}$ from \cite[Theorem 6.1.3]{YZ21} for all $n\in \Z$; here $\epsilon=2$ if $L$ is symmetric, and $\epsilon=1$ if $L$ is antisymmetric. Moreover, if $L$ is relatively ample, then  $\widehat{\L}$ is nef.
\end{thm}
We single out an important case.
\begin{thm}[adelic prolongation of the Poincaré bundle]\label{adelic prolongation of Poincare bundle}
	Let $\Sbar$ be the smooth \emph{projective} compactification of $S$ over $K$. Let $\PP$ be the Poincaré bundle associated to $A$. The extended Poincaré bundle $\overline{\PP}$ (Proposition \ref{prolongation of  Poincaré  bundle})  admits an integrable adelic extension $\widehat{\overline{\PP}}$ on $\NN(A)^0\times_{\Sbar} \NN(A^\vee)^0$, which is unique with the dynamical property 
	\[
	[n]^*\widehat{\overline{\PP}}\cong n^2 \widehat{\overline{\PP}}\quad \text{for all } n\in \Z;
	\]  in particular, the adelic extension $\widehat{\overline{\PP}}$ is rigidified.
\end{thm}
\section{Variation of the  Néron--Tate height}\label{Application to Silverman's Limit Theorem}

Let $K$ be a number field, $S$  a smooth curve over $K$ and $A/S$  an abelian scheme. Let $S\subset \Sbar$ be the smooth projective compactification over $K$. Let $\eta\in S$ be the generic point. Recall the Néron--Tate height pairing in \S\ref{Sect:Néron--Tate Height Pairing}.
\begin{thm}[Silverman's limit theorem]\label{Silverman's limit thm}
For sections $P\in A(S), Q\in A^\vee(S)$,	one has the  asymptotic description of the fiberwise Néron--Tate height pairing:
	\[
	\lim\limits_{s\in S(\Kbar), h_{\Sbar}(s)\to \infty} \frac{\langle P_s, Q_s\rangle_{A_s}}{h_{\Sbar}(s)}=\langle P_\eta, Q_\eta\rangle_{A_\eta},
	\] where $h_{\Sbar}$ is any Weil height on $\Sbar$ associated to a degree-one line bundle on $\Sbar$.
\end{thm}
\begin{proof}
	This is a direct reformulation of  \cite[Theorem B]{Sil83}. We will give another proof later in this section.
\end{proof}
\begin{rem}
	Suppose $h_{\Sbar,L}$ is a Weil height associated to a line bundle $L\in \Pic(\Sbar)$ with nonzero degree. Then Theorem \ref{Silverman's limit thm} still holds  with the normalized height $h_{\Sbar}\coloneq \frac{1}{\deg L}h_{\Sbar,L}$; see \cite[\S 4]{Sil83} for the argument.
\end{rem}

\begin{void}
	Let $\PP$ be the  Poincaré bundle on $A\times_S A^\vee$. By \cite[Theorem 6.1.3]{YZ21}, it admits an integrable adelic extension $\widehat{\PP}\in \widehat{\Pic}(A\times_S A/\Z)$, rigidified along the identity section. Consider a section 
	\[
	(P,Q)\colon S\to A\times_S A^\vee.
	\] Consider the integrable (arithmetic) adelic line bundle $\widehat{\mathcal{M}}\coloneq(P,Q)^*\widehat{\PP}$ on $S$. By \cite[Lemma 6.2.1]{YZ21}, see also the first part of Remark \ref{height and height pairing}, the  height $h_{\widehat{{\mathcal{M}}}}$ associated to the adelic line bundle $\widehat{{\mathcal{M}}}$ (\emph{cf}.\ \cite[\S 5.3.1]{YZ21}) gives the fiberwise Néron--Tate height:
	\begin{equation}\label{specialization, NT-height}
		h_{\widehat{{\mathcal{M}}}}(s)=\langle P_s,Q_s\rangle_{A_s},\quad s\in S(\overline{K}).
	\end{equation}
\end{void}
\begin{thm}\label{adelic refinement of Silverman's limit thm}
	The  adelic line bundle $\widehat{\mathcal{M}}$  extends to an integrable adelic $\Q$-line bundle~$\widehat{\overline{\mathcal{M}}}\in \widehat{\Pic}(\Sbar/\Z)_\Q$.
\end{thm}
\begin{proof}
	By Theorem \ref{adelic prolongation of Poincare bundle}, the extended Poincaré bundle $\overline{\PP}$ on the Néron model $\NN(A)^0\times_{\Sbar}\NN(A^\vee)^0$ admits an integrable adelic extension $\widehat{\overline{\PP}}\in \widehat{\Pic}(\NN(A)^0\times_{\Sbar}\NN(A^\vee)^0/\Z)$. On the other hand, there is a positive integer $n$ such that $(nP,nQ)$ extends to a section
	\[
	(\overline{P},\overline{Q})\colon \Sbar\to \NN(A)\times_{\Sbar}\NN(A^\vee)^0
	\] (recall Remark \ref{Rmk: order of component group is finite}). Then 
	\[
	\widehat{\overline{\mathcal{M}}}\coloneq \frac{1}{n^2}(\overline{P},\overline{Q})^*\widehat{\overline{\PP}}\in \widehat{\Pic}(\Sbar/\Z)_\Q
	\] is the desired extension; the construction is clearly independent of $n$. The integrability of $\widehat{\overline{\mathcal{M}}}$ follows from the integrability of $\widehat{\L}$, as it is preserved under pullback, see \cite[\S 2.5.5]{YZ21}.
\end{proof}
\begin{rem}\label{geometric part of the adelic extension}
	Let $\overline{\mathcal{M}}\in \Pic(\Sbar)_\Q$ be the underlying line bundle of $\widehat{\overline{\mathcal{M}}}$. By construction, it is obtained by the pullback 
	\[
	\overline{\mathcal{M}}\cong \frac{1}{n^2}(\overline{P},\overline{Q})^*\overline{\PP}
	\] of the extended Poincaré bundle $\overline{\PP}$ (see Proposition \ref{prolongation of  Poincaré  bundle}) along the extended sections. Thus its (geometric) degree 
	\[
	\deg \overline{\mathcal{M}}=\langle P_\eta,Q_\eta\rangle_{A_\eta}\in \Q
	\] gives the Néron--Tate height pairing of the generic sections over the function field of $\Sbar$. 
\end{rem}
We explain how   Theorem \ref{adelic refinement of Silverman's limit thm} recovers the limit theorem of Silverman.
\begin{proof}[Proof of Theorem \ref{Silverman's limit thm}]
	In fact, we will prove a finer asymptotic result. For a line bundle~$L$ on $\Sbar$, we choose a Weil height function $h_L$ attached to $L$.  Let $H$ be a line bundle of degree one on $\Sbar$. Recall the Landau symbols from Convention \ref{General Conventions}(2). We claim that the following asymptotics holds: 
	\begin{equation}\label{Silverman's limit thm, refined}
		h_{\widehat{\overline{\mathcal{M}}}}(s)=\langle P_\eta,Q_\eta\rangle h_{H}(s)+O\left(\sqrt{h_{H}(s)}+1\right),\quad s\in \Sbar(\overline{K}).
	\end{equation}
	This expression implies Silverman's limit theorem by noticing $h_{\widehat{\overline{\mathcal{M}}}}(s)=h_{\widehat{\mathcal{M}}}(s)$ for $s\in S(\overline{K})$.
	
	Now we prove \eqref{Silverman's limit thm, refined}. Write $\mu\coloneq  \deg(\overline{\mathcal{M}})=\langle P_\eta,Q_\eta\rangle_{A_\eta}$. Consider the line bundle
	\[
	L\coloneq \overline{\mathcal{M}}-(\mu-1)H.
	\] For degree reasons, $L-H\in \Pic^0(\Sbar)$, i.e.\ $L$ is algebraically equivalent to $H$. Since $H$ is ample, by \cite[Corollary 9.3.10]{BG06}, we have
	\[
	h_L=h_H+O\left(\sqrt{|h_H|}+1\right)
	\] as functions on $\Sbar(\overline{K})$. On the other hand, by Weil's height machine, we have
	\[
	h_{\overline{\mathcal{M}}}=h_L+(\mu-1)h_H+O(1).
	\] Thus
	\[
	h_{\overline{\mathcal{M}}}=\mu h_H+O\left(\sqrt{|h_H|}+1\right).
	\] We are done by noting
	\[
	h_{\overline{\mathcal{M}}}=h_{\widehat{\overline{\mathcal{M}}}}+O(1). \qedhere
	\] 
\end{proof}

\begin{rem}\label{Silverman-Tate, literature}
	The limit theorem of Silverman has been refined by several authors.	For families of elliptic curves, Tate \cite{Tate83} proved  a stronger result that the fiberwise Néron--Tate height is induced by a Weil height on the base with bounded error terms (a ``non-adelic version'' of Theorem \ref{adelic refinement of Silverman's limit thm}). Later, W. Green \cite{Green} generalized Tate's result for  families of general abelian varieties. Green's result was obtained independently by Biesel--Holmes--de Jong \cite{BHdJ17} by more elementary (without using the theory of compactification of the moduli space of abelian varieties), or more combinatorial, methods.
	
	The adelic refinement of Silverman's  theorem, namely, that the fiberwise Néron--Tate height (pairing) is induced by a height function on the base \emph{without error terms} (Theorem \ref{adelic refinement of Silverman's limit thm}),  was first established by DeMarco--Mavraki \cite{DM20} for families of elliptic curves using equidistribution techniques, and later by A. Carney \cite{Carney} for general abelian varieties (parametrized by a curve) using Yuan--Zhang's theory. However, it is unclear to us whether Carney constructs an adelic structure on the extended Poincaré bundle, as we do; see his proof of Theorem 4.1 in loc.\ cit..
\end{rem}
\begin{rem}[height and height pairing]\label{height and height pairing}
	Let $\theta$ be a  rigidified symmetric relatively ample line bundle on the abelian scheme $A/S$. By \cite[Theorem 6.1.3]{YZ21}, it has an adelic extension $\widehat{\theta}\in \widehat{\Pic}(A/\Z)$. Moreover,
	\[
	\hat{h}_{\theta}=h_{\widehat{\theta}}
	\] where the left hand side notate  the (classical) Néron--Tate height associated to $\theta$ (Definition \ref{def. NT height}). Let $P\in A(S)$ be a section and consider $$Q\coloneq \tau_P^*\theta-\theta\in \Pic^0_{A/S}(S)=A^\vee(S),$$ where $\tau_P\colon A\to A$ denotes the translate-by-$P$ morphism. The Néron--Tate height and Néron--Tate height pairing are related by
	\[
	2\hat{h}_{\theta}(s)=\langle P_s,Q_s\rangle\in \R, \quad s\in S(\Kbar),
	\] as explained in \eqref{2times Neron-Tate}. 
	
	Furthermore, by Theorem \ref{adelic extension summary, arithmetic case}, the adelic line bundle $\widehat{\theta}$ extends to a  \emph{nef} adelic line bundle $\widehat{\overline{\theta}}$ on $\NN(A)^0$. Let $\overline{P}\in \NN(A)^0(\Sbar)$ be the section extending a positive multiple  $nP\in A(S)$. Let $\widehat{\overline{\mathcal{M}}}=\frac{1}{n^2}\overline{P}^*\widehat{\overline{\theta}}\in \Pic(\Sbar)_\Q$. By the earlier discussion,
	\[
	h_{\widehat{\overline{\mathcal{M}}}}(s)=\hat{h}_\theta(P_s),\quad s\in S(\overline{K}).
	\] And the geometric Néron--Tate height  of $P_\eta$ is given by
	\[
	\hat{h}(P_\eta)\coloneq \hat{h}_{A_\eta,\theta_\eta}(P_\eta)=\deg \overline{\mathcal{M}},
	\] where $\theta_\eta$ is the symmetric ample line bundle on $A_\eta$ induced by $\theta$. We note that by \cite[Theorem 9.2.7(c)]{BG06},
	\begin{equation}\label{NT height of ample class is nonnegative}
		\hat{h}(P_\eta)\geq 0.
	\end{equation}
\end{rem}

We also recover \cite[Theorem C]{Sil83}:
\begin{cor}\label{uniform torsion, 1d}
	Let $P\in A(S)$ be a section and  assume that its geometric height is nonzero: $\hat{h}(P_\eta)\neq 0$. Let $\delta$ be a positive integer. Then the set
	\[
	T(\delta)\coloneq \{s\in S(\Kbar)\mid [\kappa(s):K]\leq \delta\text{ and }P_s\text{ is an torsion element of }A_s(\Kbar)\}
	\] is not Zariski dense  in $S$; equivalently, $T(\delta)$ is finite.
\end{cor}
\begin{proof} 
	Let $$\widehat{\overline{\mathcal{M}}}=\frac{1}{n^2}\overline{P}^*\widehat{\overline{\theta}}$$ be the nef $\Q$-adelic line bundle on $\Sbar$ introduced in  Remark \ref{height and height pairing}.  The same remark implies that $$h_{\widehat{\overline{\mathcal{M}}}}(s)=\hat{h}_\theta(P_s)+O(1)$$ as functions on $S(\overline{K})$. 
	
	By assumption and \eqref{NT height of ample class is nonnegative}, we  have $\deg(\overline{\mathcal{M}})>0$, i.e.\ $\overline{\mathcal{M}}$ is ample. The corollary then follows from the Northcott property for $h_\mathcal{M}$ (see \cite[Theorem 2.4.9]{BG06}).
\end{proof}
\begin{rem}[{\cite[Remark 1.2]{Carney}}]
	By \cite[Theorem 9.15]{Con06}, the condition $\hat{h}(P_\eta)\neq 0$ in Corollary \ref{uniform torsion, 1d} is equivalent to  $P_\eta$ \emph{not} lying in a torsion coset of the $K(S)/K$-trace of the abelian variety $A_\eta$. In particular, if the trace is zero, the condition is equivalent to  $P_\eta$ being non-torsion.
\end{rem}
\begin{rem}
	Let us mention that Silverman's original proof of his  limit theorem is based on an inequality now known as the Silverman--Tate height inequality; see \cite[Theorem~A]{Sil83}, or our joint work  \cite{BC24} with D. Biswas for some generalizations. We remark that the inequality itself holds over a higher dimensional base. A recent application of this inequality appears in the work of Dimitrov--Gao--Habegger on the uniform Mordell problem \cite{DGH}. On the other hand, as observed by Holmes \cite[Theorem 2.4]{Holmes_UniformTorsion}, a higher dimensional analogue of Corollary \ref{uniform torsion, 1d} would imply the \emph{uniform torsion conjecture} for abelian varieties.
\end{rem}
\bibliographystyle{alpha}
\bibliography{RefforPhDStudy}
\end{document}